\documentclass[11pt]{amsart}
\usepackage{amsmath}
\usepackage{amssymb}
\usepackage{amscd}
\usepackage{enumerate}
\usepackage{verbatim}
\usepackage{color}
\usepackage{bm}
\usepackage{easybmat} 
\usepackage{currfile}

\usepackage{rotating}

\input xy
\xyoption{all}

\numberwithin{equation}{section}
\newtheorem{Thm}{Theorem}[section]
\newtheorem{Lemma}[Thm]{Lemma}

\newtheorem{Prop}[Thm]{Proposition}
\theoremstyle{definition}

\newtheorem{Rmk}[Thm]{Remark}

\def\bfx{\mathbf{x}}
\def\bfy{\mathbf{y}}

\def\frakg{\mathfrak{g}}

\def\frakB{\mathfrak{B}}

\def\frakL{\mathfrak{L}}
\def\frakM{\mathfrak{M}}

\def\bbr{\mathbb{R}}

\def\calB{\mathcal{B}}

\def\x{\times}
\def\bs{\backslash}

\def\aut{\mathrm{Aut}}

\def\id{\mathrm{id}}

\def\ad{\mathrm{ad}}

\def\diag{\mathrm{diag}}

\def\O{\mathrm{O}}

\def\R{\mathrm{(R)}}

\def\ric{\mathrm{ric}}
\def\Ric{\mathrm{Ric}}

\def\boxit#1{\vbox{\hrule\hbox{\vrule\kern3pt
     \vbox{\kern3pt#1\kern3pt}\kern3pt\vrule}\hrule}}

\begin{document}
\title[Lorentzian metrics]
{Left invariant Lorentzian metrics and curvatures on non-unimodular Lie groups of dimension three}

\author[K.~Y.~Ha]{Ku~Yong~Ha}
\address{Department of mathematics, Sogang University, Seoul 04107, KOREA}
\email{kyha@sogang.ac.kr}

\author[J.~B.~Lee]{Jong~Bum~Lee}
\address{Department of Mathematics, Sogang University, Seoul 04107, KOREA}
\email{jlee@sogang.ac.kr}


\begin{abstract}
For each connected and simply connected three-dimensional non-unimodular Lie group,
we classify the left invariant Lorentzian metrics up to automorphism,
and study the extent to which curvature can be altered
by a change of metric. Thereby we obtain the Ricci operator, the scalar curvature, and the sectional
curvatures as functions of left invariant Lorentzian metrics on the three-dimensional non-unimodular Lie groups.

Our study is a continuation and extension of the previous studies done in \cite{HL2009_MN} for Riemannian metrics on three-dimensional Lie groups and in \cite{BC}
for Lorentzian metrics on three-dimensional unimodular Lie groups.
\end{abstract}

\maketitle

\section{Introduction}

Let $G$ be a connected and simply connected, three-dimensional Lie group.
The classification of all left invariant Riemannian metrics on $G$
up to automorphism of $G$ is completely carried out in \cite{HL2009_MN}.
For the Lorentzian case, a complete classification
is carried out in \cite{BC} when $G$ is unimodular.

We continue and extend the study in \cite{BC} to all non-unimodular Lie groups $G$.
So, our basic references are \cite{HL2009_MN} and \cite{BC}.
In this article, we will be concerned with two main problems:
\begin{enumerate}
\item 
to classify all the left invariant Lorentzian metrics on each connected,
simply connected three-dimensional non-unimodular Lie group
$G$ up to automorphism, and 
\item to study the extent to which curvature can be altered by a
change of left -invariant Lorentzian metric.
\end{enumerate}

There are uncountably many nonisomorphic, connected and simply connected three-dimensional
non-unimodular Lie groups. These are all
solvable and of the form $\bbr^2\rtimes_\varphi \bbr$ where $\bbr$
acts on $\bbr^2$ via a linear map $\varphi$, see Section~\ref{sec: Lie algebras}.
Let $G$ be such a Lie group.
Let $\frakM(G)$ be the space of left invariant Lorentzian metrics on $G$.
Then there is a natural action of $\aut(G)$ on the space $\frakM(G)$.
Our first goal is determine the moduli space $\aut(G)\bs\frakM(G)$.

Let $h\in \frakM(G)$ and let $\frakB$ be an orthonormal basis for the Lie algebra $\frakg$ of $G$
with respect to $h$.
By \cite{Milnor}, there is a unique linear transformation $L:\frakg\to\frakg$ satisfying
the formula $[u,v]=L(u\x v)$ for all $u,v\in\frakg$, and
$G$ is unimodular if and only if such a linear transformation $L$ is self adjoint,
i.e., $h(L(u),v)=h(u,L(v))$ for all $u,v\in\frakg$.
In this case, the matrix $[L]_{\frakB}$ is Lorentzian symmetric.
For each $h\in \frakM(G)$,
let $\frakL(\frakg)_h$ denote the space of all Lorentzian symmetric matrices $[L]_{\frakB}$.
And let $\frakL(\frakg)$ be the union of all $\frakL(\frakg)_h$.
The authors in \cite{BC} give a complete computation of the moduli space $\aut(G)\bs\frakM(G)$
using a bijection between $\aut(G)\bs\frakM(G)$ and $\frakL(\frakg)/O(2,1)$.
However, this procedure is possible only when $G$ is unimodular.
In this article, for this reason, we shall give a complete computation of the moduli space $\aut(G)\bs\frakM(G)$
in a direct way  since $G$ is non-unimodular.

In Section~\ref{Lorentzian metric}, we review some properties related to the left invariant Lorentzian metric on a connected Lie group.
In Section~\ref{sec: Lie algebras}, we recall the three-dimensional non-unimodular Lie algebras and their groups of automorphisms (\cite{HL2009_MN}).

In Section~\ref{curvature-Lorentzian},
we review Ricci operator and curvatures, sectional curvatures, and scalar curvatures of the left invariant Lorentzian metric on a connected Lie group. In particular, we find out the relationship between the Ricci curvature and the Ricci operator. Furthermore, we find an identity which is the Lorentzian version of the formula given by (\cite[p.~\!306]{Milnor}) about sectional curvature $\kappa(u,v)$ associated with $u$ and $v$.

In Sections~\ref{sec:G_I}, \ref{sec:c>1}, \ref{sec:c=1} and \ref{sec:c<1},
we classify all the left invariant Lorentzian metrics on the three-dimensional non-unimodular Lie group up to automorphism 
and we study the extent to which curvature can be altered by a change of left invariant Lorentzian metric.
In particular, we classify explicitly three-dimensional Lorentzian non-unimodular Lie groups whose metrics have constant section curvatures.

Tables~\ref{tab1}, \ref{tab2}, \ref{tab3}, \ref{tab4} and \ref{tab5} below summarize the main results of this article.
All the calculations were done using the program MATHEMATICA and hand-checked.

{\footnotesize {\scriptsize
\begin{sidewaystable} \vspace{7cm}
\caption{The case of $G_I$} \label{tab1}
\begin{center}
\begin{minipage}{\textwidth}
\begin{tabular}{|p{1.3in}|p{2.in}|p{1.in}|p{1.2in}|}\hline
Lie algebra &\multicolumn{3}{|c|}
{$\frakg_I  \cong  \bbr^2\rtimes_{\sigma_I}\bbr$, \, where \,
$\sigma_I(t) = \left[ \begin{matrix} t \!\!\!&\!\!\! 0 \\ 0
\!\!\!&\!\!\! {t} \end{matrix} \right]$}

\\ \hline
$\begin{array}{l}\hspace{-5pt}
\text{Associated connected,}\\\hspace{-5pt}
\text{simply connected}\\\hspace{-5pt}
\text{Lie group}
\end{array}$
&\multicolumn{3}{|c|}{$G_I \cong \bbr^2\rtimes_{\varphi_I}\bbr$, \, where \,
$\varphi_I(t)= \left[ \begin{matrix} e^t & 0 \\ 0 & {e^t} \end{matrix} \right]$}

\\ \hline
$\begin{array}{l}\hspace{-5pt}
\text{Automorphisms of}\\\hspace{-5pt}
\text{a Lie algebra}
\end{array}$
&\multicolumn{3}{|c|}{$\left\{\left[\begin{matrix} \mathrm{GL}(2,\bbr)& *\\ 0 & 1\end{matrix}\right]\bigg|\,\, *\in\bbr^2\right\}$}

\\ \hline\hline
$\begin{array}{l}\hspace{-5pt}
\text{Left invariant}\\\hspace{-5pt}
\text{Lorentzian metrics}
\end{array}$
&Ricci operators $[\Ric]_{\{y_i\}}$
&scalar curvature
&sectional curvatures
\\\hline
$\left[\begin{matrix} 1 & \hspace{8pt}0 & 0 \\ 0 & -1 & 0 \\ 0 & \hspace{8pt}0 & \mu \end{matrix}\right],
    \begin{array}{l} \mu > 0 \end{array} $
&$
-\left[\begin{matrix} \frac{2}{\mu} & 0 & 0 \\ 0 & \frac{2}{\mu} & 0 \\ 0 & 0 & \frac{2}{\mu} \end{matrix} \right]
\quad\text{O'Neill type $\{11,1\}$}$
&$\rho = -\frac{6}{\mu}$
&$\begin{array}{l}
\kappa(y_1,y_2)=-\frac{1}{\mu}\\
\kappa(y_2,y_3)=-\frac{1}{\mu}\\
\kappa(y_3,y_1)=-\frac{1}{\mu}
\end{array}$

\\ \hline
$\left[\begin{matrix} 1 & 0 & \hspace{8pt}0 \\ 0 & 1 & \hspace{8pt}0 \\ 0 & 0 & -\mu \end{matrix}\right],
    \begin{array}{l} \mu > 0 \end{array} $
&$\left[\begin{matrix} \frac{2}{\mu} & 0 & 0 \\ 0 & \frac{2}{\mu} & 0 \\ 0 & 0 & \frac{2}{\mu} \end{matrix} \right],
\quad\text{O'Neill type $\{11,1\}$}$
&$\rho = \frac{6}{\mu}$
&$\begin{array}{l}
\kappa(y_1,y_2)=\frac{1}{\mu}\\
\kappa(y_2,y_3)=\frac{1}{\mu}\\
\kappa(y_3,y_1)=\frac{1}{\mu}
\end{array}$

\\ \hline
$\left[\begin{matrix} 1 & 0 & 0 \\ 0 & 0 & 1 \\ 0 & 1 & 0 \end{matrix}\right]$
&$\left[\begin{matrix} 0 & 0 & 0 \\ 0 & 0 & 0 \\ 0 & 0 & 0 \end{matrix} \right]$
\quad(flat), \quad O'Neill type $\{11,1\}$
&$\rho = 0$
&$\begin{array}{l}
\kappa(y_1,y_2)=0\\
\kappa(y_2,y_3)=0\\
\kappa(y_3,y_1)=0
\end{array}$
\\ \hline
\end{tabular}
\end{minipage}
\end{center}
\end{sidewaystable}
}}

{\footnotesize {\scriptsize
\begin{sidewaystable} \vspace{7cm}

\caption{The case of $G_c$, $c>1$} \label{tab2}
\begin{center}
\begin{minipage}{\textwidth}
\begin{tabular}{|p{1.3in}|p{2.9in}|p{1.25in}|p{2in}|}\hline
Lie algebra & \multicolumn{3}{|c|}
{$\frakg_c \cong \bbr^2\rtimes_{\sigma_c}\bbr$, \, where \,
$\sigma_c(t) \!\!=\!\! \left[ \begin{matrix} 0 &\!\!\! -ct
\\ t & \!\!\!\hspace{8pt}{2t} \end{matrix} \right]$}

\\ \hline
$\begin{array}{l}\hspace{-5pt}
\text{Associated connected,}\\\hspace{-5pt}
\text{simply connected}\\\hspace{-5pt}
\text{Lie group}
\end{array}$
& \multicolumn{3}{|c|}
{$G_c  \cong \bbr^2\rtimes_{\varphi_c}\bbr$, \,
where \,
$\varphi_c(t)= e^t\frac{e^{zt}+e^{-zt}}{2}\left[ \begin{matrix} 1 \!\!&\!\! 0 \\ 0 \!\!&\!\! 1 \end{matrix} \right]
+e^t\frac{e^{zt}-e^{-zt}}{2z}\left[ \begin{matrix} -1 \!\!&\!\!\! -c \\ \hspace{8pt}1 \!\!&\!\!\! \hspace{8pt}1 \end{matrix} \right], \quad
  (w=\sqrt{1-c}\ne 0)$}

\\ \hline
$\begin{array}{l}\hspace{-5pt}
\text{Automorphisms of}\\\hspace{-5pt}
\text{a Lie algebra}
\end{array}$
& \multicolumn{3}{|c|}
{$\Bigg\{\Bigg[\begin{matrix} \beta\!\!-\!\!\alpha \!\!&\!\! -c\alpha \!\!&\!\! *\\
                               \alpha \!\!&\!\! \beta\!\!+\!\!\alpha \!\!&\!\! *  \\
                            0 \!\!&\!\! 0 \!\!&\!\! 1\end{matrix} \Bigg]\bigg|\,
  \begin{array}{l} \alpha,\beta, * \in\bbr, \\\beta^2 + (c-1)\alpha^2 \ne 0 \end{array} \!\!\Bigg\}$}

\\ \hline\hline
$\begin{array}{l}\hspace{-5pt}
\text{Left invariant}\\\hspace{-5pt}
\text{Lorentzian metrics}
\end{array}$
&Ricci operators $[\Ric]_{\{y_i\}}$
&scalar curvature
&sectional curvatures
\\\hline
$\left[\begin{matrix} \mu & 0 & 0\\ 0 & 0 & 1\\ 0 & 1 & 0\end{matrix}\right],
    \hspace{-.1cm} \begin{array}{l} \mu > 0 \end{array} $
&$\left[\begin{matrix} -\frac{c^2\mu}{2} & 0 & \hspace{8pt}0 \\ \hspace{8pt}0 & \frac{c(c\mu+1)}{2} & -\frac{c}{2} \\ \hspace{8pt}0 & \frac{c}{2} & \frac{c(c\mu-1)}{2} \end{matrix} \right], \quad
\text{O'Neill type $\{21\}$}$
&$\rho = \frac{c^2\mu}{2}$
&$\begin{array}{l}
\kappa(y_1,y_2)=-\frac{c(c\mu-2)}{4}\\
\kappa(y_2,y_3)=\frac{3c^2\mu}{4}\\
\kappa(y_3,y_1)=-\frac{c(c\mu+2)}{4}
\end{array}$

\\ \hline
$\left[\begin{matrix} 1 & 1 & 0\\ 1 & \tau & 0\\ 0 & 0 & \mu\end{matrix}\right],
    \hspace{-.1cm} \begin{array}{l} \mu > 0  \\ \tau <1 \end{array} $
&\hspace{-5pt}$\begin{array}{l}
\left[\begin{matrix} \frac{\tau^2+(4-2c)\tau+(c^2-4)}{2(1-\tau)\mu} & 0 & 0 \\
0 &\hspace{-15pt} \frac{\tau^2+2\tau-(c^2-2c+4)}{2(1-\tau)\mu} & -\frac{c-\tau}{\mu\sqrt{1-\tau}} \\
0 & \frac{c-\tau}{\mu\sqrt{1-\tau}} &\hspace{-15pt}
  -\frac{\tau^2-6\tau-(c^2-2c-4)}{2(1-\tau)\mu} \end{matrix} \right]\\
\text{O'Neill type $\{11,1\}$, $\{1z\bar{z}\}$, $\{21\}$} \\ \qquad
\text{when $(c+\tau)^2-4c >0, <0, =0$,  resp.}\end{array}
$
&$\rho = \frac{\tau^2-2(c-6)\tau+(c^2-12)}{2(1-\tau)\mu}$
&$\begin{array}{l}
\kappa(y_1,y_2)=\frac{3\tau^2-2c\tau-(c^2-4c+4)}{4(1-\tau)\mu}\\
\kappa(y_2,y_3)=-\frac{\tau^2-2(c+2)\tau+(c^2+4)}{4(1-\tau)\mu}\\
\kappa(y_3,y_1)=-\frac{\tau^2+2(c-4)\tau-(3c^2-4c-4)}{(1-\tau)\mu}
\end{array}$

\\ \hline
$\left[\begin{matrix} 1 & 1 & \hspace{8pt}0\\ 1 & \nu & \hspace{8pt}0\\ 0 & 0 & -\mu\end{matrix}\right],
    \hspace{-.1cm} \begin{array}{l} \mu >0 \\ 1<\nu\le c\end{array} $
&\hspace{-5pt}$\begin{array}{l}
\left[\begin{matrix} -\frac{\nu^2+2\nu-(c^2-2c+4)}{2(\nu-1)\mu} & \frac{c-\nu}{\mu\sqrt{\nu-1}} & 0 \\\frac{c-\nu}{\mu\sqrt{\nu-1}} &\hspace{-18pt} -\frac{\nu^2-6\nu-(c^2-2c-4)}{2(\nu-1)\mu} & 0 \\
0 & 0 &\hspace{-18pt} \frac{(\nu-c)^2+4(\nu-1)}{2(\nu-1)\mu} \end{matrix} \right] \\
\text{O'Neill type $\{11,1\}$}
\end{array}$
&$\rho = \frac{(\nu-c)^2+12(\nu-1)}{2(\nu-1)\mu}$
&$\begin{array}{l}
\kappa(y_1,y_2)=-\frac{\nu^2-(2c+4)\nu+(c^2+4)}{4(\nu-1)\mu}\\
\kappa(y_2,y_3)=-\frac{\nu^2+2(c-4)\nu-(3c^2-4c-4)}{4(\nu-1)\mu}\\
\kappa(y_3,y_1)=\frac{3\nu^2-2c\nu-(c^2-4c+4)}{4(\nu-1)\mu}
\end{array}$

\\ \hline
\end{tabular}
\end{minipage}
\end{center}
\end{sidewaystable}
}}

{\footnotesize {\scriptsize
\begin{sidewaystable} \vspace{10cm}

\caption{The case of $G_1$} \label{tab3}
\begin{center}%
\begin{minipage}{\textwidth}\hspace{-3cm}
\begin{tabular}{|p{1.5in}|p{3.5in}|p{1.in}|p{1.5in}|}\hline
Lie algebra & \multicolumn{3}{|c|}
{$\frakg_1 \cong \bbr^2\rtimes_{\sigma_1}\bbr$, \, where \,
$\sigma_1(t) \!\!=\!\! \left[ \begin{matrix} 0 &\!\!\! -t
\\ t &\!\! {2t} \end{matrix} \right]$}

\\ \hline
$\begin{array}{l}\hspace{-5pt}
\text{Associated connected,}\\\hspace{-5pt}
\text{simply connected}\\\hspace{-5pt}
\text{Lie group}
\end{array}$
& \multicolumn{3}{|c|}
{$\begin{array}{c} G_1  \cong \bbr^2\rtimes_{\varphi_1}\bbr, \\
\quad \text{where \,}
\varphi_1(t)= e^t\left[ \begin{matrix} 1 \!\!&\!\! 0 \\ 0 \!\!&\!\! {1} \end{matrix} \right]+e^tt\left[\begin{matrix} -1\!\!\! &\!\!\! -1 \\ 1\!\!\! &\!\!\! 1\end{matrix} \right]\end{array}$}

\\ \hline
$\begin{array}{l}\hspace{-5pt}
\text{Automorphisms of}\\\hspace{-5pt}
\text{a Lie algebra}
\end{array}$
&\multicolumn{3}{|c|}
{$\Bigg\{\Bigg[\begin{matrix} \beta\!\!-\!\!\alpha \!\!&\!\! -c\alpha \!\!&\!\! *\\
                               \alpha \!\!&\!\! \beta\!\!+\!\!\alpha \!\!&\!\! *  \\
                            0 \!\!&\!\! 0 \!\!&\!\! 1\end{matrix} \Bigg]\bigg|\,
  \begin{array}{l} \alpha,\beta, * \in\bbr, \\\beta^2 + (c-1)\alpha^2 \ne 0 \end{array} \!\!\Bigg\}$}

\\ \hline\hline
$\begin{array}{l}\hspace{-5pt}
\text{Left invariant}\\\hspace{-5pt}
\text{Lorentzian metrics}
\end{array}$
&Ricci operators $[\Ric]_{\{y_i\}}$
&scalar curvature
&sectional curvatures

\\ \hline
$Q^{-1}\left[\begin{matrix} 0 & 0 & 1\\ 0 & \mu & 0\\ 1 & 0 & 0\end{matrix}\right]Q,
    \hspace{-.1cm} \begin{array}{l} \mu > 0 \end{array} $
&$
\diag\{0,0,0\} \quad(flat), \quad\text{O'Neill type  $\{11,1\}$}$
&$\rho = 0$
&$\begin{array}{l}
\kappa(y_1,y_2)=0 \\
\kappa(y_2,y_3)=0 \\
\kappa(y_3,y_1)=0
\end{array}$

\\ \hline
$Q^{-1}\left[\begin{matrix} \mu & 0 & 0\\ 0 & 0 & 1\\ 0 & 1 & 0 \end{matrix}\right]Q,
    \hspace{-.1cm} \begin{array}{l} \mu > 0 \end{array} $
&$\left[\begin{matrix} -\frac{\mu}{2} & \sqrt{\frac{\mu}{2}} & -\sqrt{\frac{\mu}{2}} \\
\sqrt{\frac{\mu}{2}} & \frac{\mu}{2} & 0 \\
\sqrt{\frac{\mu}{2}} & 0 & \frac{\mu}{2}  \end{matrix} \right], \quad\text{O'Neill type  $\{21\}$}$
&$\rho = \frac{\mu}{2}$
&$\begin{array}{l}
\kappa(y_1,y_2)=-\frac{\mu}{4} \\
\kappa(y_2,y_3)=\frac{3\mu}{4} \\
\kappa(y_3,y_1)=-\frac{\mu}{4}
\end{array}$

\\ \hline
$Q^{-1}\left[\begin{matrix} 1 & \hspace{8pt}0 & 0\\ 0 & -\nu & 0\\ 0 & \hspace{8pt}0 & \mu \end{matrix}\right]Q,
    \hspace{-.1cm} \begin{array}{l} \mu >0 \\ \nu>0 \end{array} $
&$
\left[\begin{matrix} -\frac{4\nu+1}{2\mu\nu} & 0 & -\frac{1}{\mu\sqrt{\nu}} \\
0 & \frac{1-4\nu}{2\mu\nu} & 0 \\
\hspace{8pt}\frac{1}{\mu\sqrt{\nu}} & 0 & \hspace{8pt}\frac{1-4\nu}{2\mu\nu} \end{matrix} \right],
\begin{array}{l}
\text{O'Neill type  $\{11,1\}$, $\{1z\bar{z}\}$, $\{21\}$}\\
\text{when $1-4\nu >0, <0, =0$, resp.}
\end{array}$
&$\rho = \frac{1-12\nu}{2\mu\nu}$
&$\begin{array}{l}
\kappa(y_1,y_2)=-\frac{1+4\nu}{4\mu\nu} \\
\kappa(y_2,y_3)= \frac{3-4\nu}{4\mu\nu}\\
\kappa(y_3,y_1)=-\frac{1+4\nu}{4\mu\nu}
\end{array}$

\\ \hline
$Q^{-1}\left[\begin{matrix} 1 & 0 & \hspace{8pt}0\\ 0 & \nu & \hspace{8pt}0\\ 0 & 0 & -\mu \end{matrix}\right]Q,
    \hspace{-.1cm} \begin{array}{l} \mu > 0  \\ \nu>0 \end{array} $
&$\left[\begin{matrix} \frac{4\nu-1}{2\mu\nu} & \frac{1}{\mu\sqrt{\nu}}& 0 \\
\frac{1}{\mu\sqrt{\nu}} & \frac{4\mu+1}{2\mu\nu} & 0 \\
0 & 0 & \frac{4\mu+1}{2\mu\nu} \end{matrix} \right], \quad\text{O'Neill type  $\{11,1\}$}$
&$\rho =\frac{1+12\nu}{2\mu\nu}> 0 $
&$\begin{array}{l}
\kappa(y_1,y_2)=\frac{4\nu-1}{4\mu\nu} \\
\kappa(y_2,y_3)=\frac{4\nu+3}{4\mu\nu} \\
\kappa(y_3,y_1)=\frac{4\nu-1}{4\mu\nu}
\end{array}$

\\ \hline
$Q^{-1}\left[\begin{matrix} -1 & 0 & 0\\ \hspace{8pt}0 & \nu & 0\\ \hspace{8pt}0 & 0 & \mu \end{matrix}\right]Q,
    \hspace{-.1cm} \begin{array}{l} \mu >0 \\ \nu>0 \end{array} $
&$
\left[\begin{matrix} \hspace{8pt}\frac{1-4\nu}{2\mu\nu} & 0 & \frac{1}{\mu\sqrt{\nu}} \\
\hspace{8pt}0   & \frac{1-4\mu}{2\mu\nu} & 0 \\
-\frac{1}{\mu\sqrt{\nu}} & 0 & -\frac{4\mu+1}{2\mu\nu} \end{matrix} \right],
\begin{array}{l}
\text{O'Neill type  $\{11,1\}$, $\{1z\bar{z}\}$, $\{21\}$}\\
\text{when $1-4\nu >0, <0, =0$, resp.}
\end{array}$
&$\rho = \frac{1-12\nu}{2\mu\nu}$
&$\begin{array}{l}
\kappa(y_1,y_2)=\frac{3-4\nu}{4\mu\nu} \\
\kappa(y_2,y_3)=-\frac{1+4\nu}{4\mu\nu}\\
\kappa(y_3,y_1)=-\frac{1+4\nu}{4\mu\nu}
\end{array}$

\\ \hline
$Q^{-1}\left[\begin{matrix} 0 & 1 & 0\\ 1 & 0 & 0\\ 0 & 0 & \mu \end{matrix}\right]Q,
    \hspace{-.1cm} \begin{array}{l} \mu > 0 \end{array} $
&$\left[\begin{matrix} -\frac{2}{\mu} &\hspace{8pt} 0 & \hspace{8pt}0 \\
\hspace{8pt}0 & -\frac{3}{\mu} & \hspace{8pt}\frac{1}{\mu} \\
\hspace{8pt}0 & -\frac{1}{\mu} & -\frac{1}{\mu}  \end{matrix} \right], \quad\text{O'Neill type  $\{21\}$}
$
&$\rho = -\frac{6}{\mu}$
&$\begin{array}{l}
\kappa(y_1,y_2)=-\frac{2}{\mu} \\
\kappa(y_2,y_3)= -\frac{1}{\mu} \\
\kappa(y_3,y_1)=0
\end{array}$

\\ \hline
$Q^{-1}\left[\begin{matrix} \hspace{8pt}0 &\!\!\! -1 & 0\\ -1 &\!\!\! \hspace{8pt}0 & 0\\ \hspace{8pt}0 & \!\!\!\hspace{8pt}0 & \mu \end{matrix}\right]Q,
    \hspace{-.1cm} \begin{array}{l} \mu > 0 \end{array} $
&$\left[\begin{matrix} -\frac{2}{\mu} & \hspace{8pt}0 & \hspace{8pt}0 \\
\hspace{8pt}0 & -\frac{1}{\mu} & -\frac{1}{\mu} \\
\hspace{8pt}0 & \hspace{8pt}\frac{1}{\mu} & -\frac{3}{\mu}  \end{matrix} \right], \quad\text{O'Neill type $\{21\}$}$
&$\rho = -\frac{6}{\mu}$
&$\begin{array}{l}
\kappa(y_1,y_2)=0 \\
\kappa(y_2,y_3)= -\frac{1}{\mu} \\
\kappa(y_3,y_1)=-\frac{2}{\mu}
\end{array}$

\\ \hline
\end{tabular}
\par \quad\hspace{-3cm} Here {\footnotesize $Q=\left[\begin{matrix}-2&-1&0\\\hspace{8pt}1&\hspace{8pt}1&0\\ \hspace{8pt}0&\hspace{8pt}0&1\end{matrix}\right]$}.
\end{minipage}
\end{center}
\end{sidewaystable}
}}

{\footnotesize {\scriptsize
\begin{sidewaystable} \vspace{10cm}

\caption{The case of $G_c$, $c<1$, Part~I} \label{tab4}
\begin{center}
\begin{minipage}{\textwidth}
\begin{tabular}{|p{1.5in}|p{3.5in}|p{1.in}|p{1.5in}|}\hline
Lie algebra & \multicolumn{3}{|c|}
{$\frakg_c \cong \bbr^2\rtimes_{\sigma_c}\bbr$, \, where \,
$\sigma_c(t) \!\!=\!\! \left[ \begin{matrix} 0 &\!\!\! -ct
\\ t &\!\! {2t} \end{matrix} \right]$}

\\ \hline
$\begin{array}{l}\hspace{-5pt}
\text{Associated connected,}\\\hspace{-5pt}
\text{simply connected}\\\hspace{-5pt}
\text{Lie group}
\end{array}$
& \multicolumn{3}{|c|}
{$\begin{array}{c} G_c  \cong \bbr^2\rtimes_{\varphi_c}\bbr, \\
\text{where \,} \\
\varphi_c(t)= e^t\frac{e^{wt}+e^{-wt}}{2}\left[ \begin{matrix} 1 \!\!&\!\! 0 \\ 0 \!\!&\!\! 1 \end{matrix} \right]
+e^t\frac{e^{wt}-e^{-wt}}{2w}\left[ \begin{matrix} -1 \!\!&\!\! -c \\ 1 \!\!&\!\! 1 \end{matrix} \right]
\qquad  (w=\sqrt{1-c}> 0)\end{array}$}

\\ \hline
$\begin{array}{l}\hspace{-5pt}
\text{Automorphisms of}\\\hspace{-5pt}
\text{a Lie algebra}
\end{array}$
& \multicolumn{3}{|c|}
{$\Bigg\{\Bigg[\begin{matrix} \beta\!\!-\!\!\alpha \!\!&\!\! -c\alpha \!\!&\!\! *\\
                               \alpha \!\!&\!\! \beta\!\!+\!\!\alpha \!\!&\!\! *  \\
                            0 \!\!&\!\! 0 \!\!&\!\! 1\end{matrix} \Bigg]\bigg|\,
  \begin{array}{l} \alpha,\beta, * \in\bbr, \\\beta^2 + (c-1)\alpha^2 \ne 0 \end{array} \!\!\Bigg\}$}

\\ \hline\hline
$\begin{array}{l}\hspace{-5pt}
\text{Left invariant}\\\hspace{-5pt}
\text{Lorentzian metrics}
\end{array}$
&Ricci operators $[\Ric]_{\{y_i\}}$
&scalar curvature
&sectional curvatures

\\ \hline
$P^{-1}\left[\begin{matrix} 0 & 0 & 1\\ 0 & 1 & 0\\ 1 & 0 & 0 \end{matrix}\right]P$
&\hspace{-5pt}$\begin{array}{l}
\left[\begin{matrix} 0 & 0 & 0 \\ 0 & -w(w-1) & w(w-1) \\ 0 & -w(w-1) & w(w-1) \end{matrix} \right],
\end{array}
\begin{array}{l}
\text{O'Neill type $\{11,1\}, \{21\}$}\\
\text{when $w=1, w\ne1$, resp.}
\end{array}$
&$\rho = 0$
&$\begin{array}{l}
\kappa(y_1,y_2)=-w(w-1) \\
\kappa(y_2,y_3)=0 \\
\kappa(y_3,y_1)=w(w-1)
\end{array}$

\\ \hline
$P^{-1}\left[\begin{matrix} 1 & 0 & 0\\ 0 & 0 & 1\\ 0 & 1 & 0 \end{matrix}\right]P$
&$\left[\begin{matrix} 0 & 0 & 0 \\ 0 & -w(w+1) & w(w+1) \\
0 & -w(w+1) & w(w+1) \end{matrix} \right], \quad\text{O'Neill type $\{21\}$}$
&$\rho = 0$
&$\begin{array}{l}
\kappa(y_1,y_2)=-w(w+1) \\
\kappa(y_2,y_3)=0 \\
\kappa(y_3,y_1)=w(w+1)
\end{array}$

\\ \hline
$P^{-1}\left[\begin{matrix} 1 & 1 & 0\\ 1 & 1 & \mu\\ 0 & \mu & 0 \end{matrix}\right]P,
    \hspace{-.1cm} \begin{array}{l} \mu >0 \end{array} $
&$\left[\begin{matrix} -\frac{2w^2}{\mu^2} & \frac{w(1+w)}{\mu^2} & -\frac{w(1+w)}{\mu^2} \\
\frac{w(1+w)}{\mu^2} & \frac{w(3w-1)}{2\mu^2} & \frac{w(1+w)}{2\mu^2} \\
\frac{w(1+w)}{\mu^2} & -\frac{w(1+w)}{2\mu^2} & \frac{z(1+5w)}{2\mu^2}\end{matrix} \right], \quad\text{O'Neill type $\{21\}$}$
&$\rho = \frac{2w^2}{\mu^2}$
&$\begin{array}{l}
\kappa(y_1,y_2)=-\frac{w(1+3w)}{2\mu^2} \\
\kappa(y_2,y_3)=\frac{3w^2}{\mu^2} \\
\kappa(y_3,y_1)=\frac{w(1-w)}{2\mu^2}
\end{array}$

\\ \hline
$P^{-1}\left[\begin{matrix} 1 & 0 & \hspace{8pt}0\\ 0 & 1 & \hspace{8pt}0\\ 0 & 0 & -\mu \end{matrix}\right]P,
    \hspace{-.1cm} \begin{array}{l} \mu > 0 \end{array} $
&$\diag\{\frac{2(1+w)}{\mu}, \frac{2(1-w)}{\mu}, \frac{2(1+w^2)}{\mu}\},
\quad\text{O'Neill type $\{11,1\}$}$
&$\rho =\frac{2(3+w^2)}{\mu} $
&$\begin{array}{l}
\kappa(y_1,y_2)=\frac{1-w^2}{\mu} \\
\kappa(y_2,y_3)= \frac{(1-w)^2}{\mu}\\
\kappa(y_3,y_1)=\frac{(1+w)^2}{\mu}
\end{array}$

\\ \hline
$P^{-1}\left[\begin{matrix} 1 & \hspace{8pt}0 & 0\\ 0 & -1 & 0\\ 0 & \hspace{8pt}0 & \mu \end{matrix}\right]P,
    \hspace{-.1cm} \begin{array}{l} \mu >0 \end{array} $
&$\diag\{-\frac{2(1+w)}{\mu}, -\frac{2(1+w^2)}{\mu}, \frac{2(1-w)}{\mu}\},
\quad\text{O'Neill type $\{11,1\}$}$
&$\rho =-\frac{2(3+w^2)}{\mu} $
&$\begin{array}{l}
\kappa(y_1,y_2)=-\frac{(1+w)^2}{\mu}\\
\kappa(y_2,y_3)=-\frac{(1-w)^2}{\mu} \\
\kappa(y_3,y_1)=-\frac{1-w^2}{\mu}
\end{array}$

\\ \hline
$P^{-1}\left[\begin{matrix} -1 & 0 & 0\\ \hspace{8pt}0 & 1 & 0\\ \hspace{8pt}0 & 0 & \mu \end{matrix}\right]P,
    \hspace{-.1cm} \begin{array}{l} \mu > 0 \end{array} $
&$\diag\{\frac{2(w-1)}{\mu}, -\frac{2(1+w^2)}{\mu}, -\frac{2(1+w)}{\mu}\},
\quad\text{O'Neill type $\{11,1\}$}$
&$\rho = -\frac{2(3+w^2)}{\mu}$
&$\begin{array}{l}
\kappa(y_1,y_2)= -\frac{(1-w)^2}{\mu}\\
\kappa(y_2,y_3)= -\frac{(1+w)^2}{\mu}\\
\kappa(y_3,y_1)=-\frac{1-w^2}{\mu}
\end{array}$

\\ \hline
\end{tabular}
\end{minipage}
\end{center}
\end{sidewaystable}
}}

{\footnotesize {\scriptsize
\begin{sidewaystable} \vspace{10cm}

\caption{The case of $G_c$, $c<1$, Part~II} \label{tab5}
\begin{center}
\begin{minipage}{\textwidth}\hspace{-3cm}
\begin{tabular}{|p{1.5in}|p{2.8in}|p{1.in}|p{2.2in}|}\hline
$\begin{array}{l}\hspace{-5pt}
\text{Left invariant}\\\hspace{-5pt}
\text{Lorentzian metrics}
\end{array}$
&Ricci operators $[\Ric]_{\{y_i\}}$
&scalar curvature
&sectional curvatures

\\ \hline
$P^{-1}\left[\begin{matrix} 0 & 1 & 0\\ 1 & 0 & 0\\ 0 & 0 & \mu \end{matrix}\right]P,
    \hspace{-.1cm} \begin{array}{l} \mu > 0 \end{array} $
&$\diag\{-\frac{2}{\mu}, -\frac{2}{\mu}, -\frac{2}{\mu}\},
\quad\text{O'Neill type $\{11,1\}$}$
&$\rho = -\frac{6}{\mu}$
&$\begin{array}{l}
\kappa(y_1,y_2)= -\frac{1}{\mu}\\
\kappa(y_2,y_3)= -\frac{1}{\mu}\\
\kappa(y_3,y_1)= -\frac{1}{\mu}
\end{array}$

\\ \hline
$P^{-1}\left[\begin{matrix} 0 & 1 & 0\\ 1 & 1 & 0\\ 0 & 0 & \mu \end{matrix}\right]P,
    \hspace{-.1cm} \begin{array}{l} \mu > 0 \end{array} $
&\hspace{-5pt}$\begin{array}{l}
\left[\begin{matrix} -\frac{2}{\mu} & 0 & 0 \\
0 & -\frac{2(w^2-w+1)}{\mu} & \frac{2w(w-1)}{\mu} \\
0 & -\frac{2w(w-1)}{\mu} & \frac{2(w^2-w-1)}{\mu} \end{matrix} \right],\\
\quad\text{O'Neill type  $\{11,1\}, \{21\}$}\\
\quad\text{when $w=1, w\ne1$, resp.}
\end{array}$
&$\rho = -\frac{6}{\mu}$
&$\begin{array}{l}
\kappa(y_1,y_2)= -\frac{2w^2-2w+1}{\mu}\\
\kappa(y_2,y_3)= -\frac{1}{\mu}\\
\kappa(y_3,y_1)=\frac{2w^2-2w-1}{\mu}
\end{array}$

\\ \hline
$P^{-1}\left[\begin{matrix} 0 & \hspace{8pt}1 & 0\\ 1 & -1 & 0\\ 0 & \hspace{8pt}0 & \mu \end{matrix}\right]P,
    \hspace{-.1cm} \begin{array}{l} \mu > 0 \end{array} $
&\hspace{-5pt}$\begin{array}{l}
\left[\begin{matrix} -\frac{2}{\mu} & 0 & 0 \\
\hspace{8pt}0 & \frac{2(w^2-w-1)}{\mu} & \frac{2w(w-1)}{\mu} \\
\hspace{8pt}0 & -\frac{2w(w-1)}{\mu} & -\frac{2(w^2-w+1)}{\mu} \end{matrix} \right],\\
\quad\text{O'Neill type $\{11,1\}, \{21\}$}\\
\quad\text{when $w=1, w\ne1$, resp.}
\end{array}$
&$\rho = -\frac{6}{\mu}$
&$\begin{array}{l}
\kappa(y_1,y_2)= \frac{2w^2-2w-1}{\mu}\\
\kappa(y_2,y_3)= -\frac{1}{\mu}\\
\kappa(y_3,y_1)=-\frac{2w^2-2w+1}{\mu}
\end{array}$

\\ \hline
$P^{-1}\left[\begin{matrix} 1 & 1 & 0\\ 1 & \tau & 0\\ 0 & 0 & \nu \end{matrix}\right]P,
    \hspace{-.1cm} \begin{array}{l} \nu > 0 \\ \tau<1 \end{array} $
& \hspace{-5pt}$\begin{array}{l}
\left[\begin{matrix} -\frac{2(w^2+w+1-(1+w)\tau)}{\nu(1-\tau)} & 0 & -\frac{2w(1+w)}{\nu\sqrt{1-\tau}} \\
0 &\hspace{-15pt} \frac{2(w^2\tau+\tau-1)}{\nu(1-\tau)} & 0 \\
\frac{2w(1+w)}{\nu\sqrt{1-\tau}} & 0 &\hspace{-15pt} \frac{2(w^2+w-1+(1-w)\tau)}{\nu\sqrt{1-\tau}} \end{matrix} \right],\\
\quad\text{O'Neill type $\{11,1\}$, $\{1z\bar{z}\}$, $\{21\}$ } \\
\quad \text{ when $\tau(\tau-1+w^2) >0, <0, =0$, resp.}
\end{array}$
&$\rho = \frac{2(w^2\tau+3\tau-3)}{\nu(1-\tau)}$
&$\begin{array}{l}
\kappa(y_1,y_2)= \frac{(w^2+2w+1)\tau-(2w^2+2w+1)}{\nu(1-\tau)}\\
\kappa(y_2,y_3)= \frac{(w^2-2w+1)\tau+(2w^2+2w-1)}{\nu(1-\tau)}\\
\kappa(y_3,y_1)=-\frac{w^2\tau-\tau+1}{\nu(1-\tau)}
\end{array}$

\\ \hline
$P^{-1}\left[\begin{matrix} 1 & 1 & 0\\ 1 & \tau & 0\\ 0 & 0 & \nu \end{matrix}\right]P,
    \hspace{-.1cm} \begin{array}{l} \nu > 0 \\ \tau>1 \end{array} $
&\hspace{-5pt}$\begin{array}{l}
\left[\begin{matrix} \frac{2(w^2+w+1-(1+w)\tau)}{\nu(\tau-1)} & -\frac{2w(1+w)}{\nu\sqrt{\tau-1}} & 0 \\
-\frac{2w(1+w)}{\nu\sqrt{\tau-1}} &\hspace{-18pt} -\frac{2(w^2+w-1+(1-w)\tau)}{\nu(\tau-1)} & 0 \\
0 & 0 &\hspace{-18pt} -\frac{2(w^2\tau+\tau-1)}{\nu(\tau-1)}\end{matrix} \right],\\
\quad\text{O'Neill type $\{11,1\}$} \end{array}$
&$\rho = -\frac{2(w^2\tau+3\tau-3)}{\nu(\tau-1)}$
&$\begin{array}{l}
\kappa(y_1,y_2)= \frac{(w^2\tau-\tau+1)}{\nu(\tau-1)}\\
\kappa(y_2,y_3)= -\frac{(w^2-2w+1)\tau+(2w^2+2w-1)}{\nu(\tau-1)}\\
\kappa(y_3,y_1)=-\frac{(w^2+2w+1)\tau-(2w^2+2w+1)}{\nu(\tau-1)}
\end{array}$

\\ \hline
$P^{-1}\left[\begin{matrix} -1 &\!\!\! \hspace{8pt}1 &\!\!\! 0\\ \hspace{8pt}1 &\!\!\! -\eta &\!\!\! 0\\ \hspace{8pt}0 &\!\!\! \hspace{8pt}0 &\!\!\! \mu \end{matrix}\right]P,
    \begin{array}{l} \mu > 0 \\ \eta<1 \end{array}$
&\hspace{-5pt}$\begin{array}{l}
\left[\begin{matrix} \frac{2(w^2\eta+\eta-1)}{\mu(1-\eta)} & 0 & 0 \\
0 &\hspace{-15pt} \frac{2(w^2+w-1+\eta-w\eta)}{\mu(1-\eta)} & \frac{2w(1+w)}{\mu\sqrt{1-\eta}} \\
0 & -\frac{2w(1+w)}{\mu\sqrt{1-\eta}} &\hspace{-15pt} -\frac{2(w^2+w+1 - (1+w)\eta)}{\mu(1-\eta)} \end{matrix} \right], \\
\quad\text{O'Neill type $\{11,1\}$, $\{1z\bar{z}\}$, $\{21\}$}\\
\quad \text{when $\eta(\eta-1+w^2) >0, <0, =0$, resp.}
\end{array}$
&$\rho = \frac{2(w^2\eta+3\eta-3)}{\mu(1-\eta)}$
&$\begin{array}{l}
\kappa(y_1,y_2)=\frac{(w^2-2w+1)\eta+(2w^2+2w-1)}{\mu(1-\eta)} \\
\kappa(y_2,y_3)=-\frac{1-\eta+w^2\eta}{\mu(1-\eta)} \\
\kappa(y_3,y_1)=\frac{(w^2+2w+1)\eta-(2w^2+2w+1)}{\mu(1-\eta)}
\end{array}$

\\ \hline
\end{tabular}
\par \quad\hspace{-3cm} Here {\footnotesize $P=\left[\begin{matrix}\frac{1}{2z}&\frac{1+z}{2z}&0\\-\frac{1}{2z}&\frac{z-1}{2z}&0\\ 0&0&1\end{matrix}\right]$}.
\end{minipage}
\end{center}
\end{sidewaystable}
}}

\section{Preliminaries}
\subsection{Left invariant Lorentzian metrics on Lie groups}\label{Lorentzian metric}

Let $h$ be a Lorentzian metric on a connected Lie group $G$,
and let $\theta: G \to G$ be a diffeomorphism on $G$. Then $\theta$
induces a Lorentzian metric on $G$ by the rule $h_\theta(x,y)=h(\theta_*^{-1}(x), \theta_*^{-1}(y))$,
where $\theta_*$ is the differential of $\theta$.
Even though $h$ is left invariant,
the induced metric $h_\theta$ is not necessarily left invariant.

Now we describe all the left invariant Lorentzian metrics
on a connected Lie group of dimension $n$.
Fix a basis $\frakB=\{x_1, x_2, \cdots, x_n\}$ for the Lie algebra $\frakg$ of $G$
and let $\{\omega_1, \omega_2,\cdots, \omega_n\}$ be its dual basis.
Then every left invariant Lorentzian metric $h$ on $G$ is of the form
$$
h = \sum h_{ij}~\omega_i\otimes \omega_j.
$$
This yields a symmetric, nondegenerate matrix $[h]=[h]_\frakB=[h_{ij}]$
satisfying
$$
h(x,y)=[x]^t[h][y]
$$
where $[x]=[x]_\frakB$ is the column vector $[a_1\,a_2 \cdots a_n]^t$ if and
only if $x=a_1 x_1+a_2 x_2+\cdots +a_n x_n$.

Let $\varphi\in\aut(\frakg)$ and let $h$ be a left invariant Lorentzian
metric on $G$. Define $h_\varphi$ by $h_\varphi(x,y) =
h(\varphi^{-1}(x), \varphi^{-1}(y))$ for all $x,y\in\frakg$. Then
$h_\varphi$ is a left invariant Lorentzian metric on $G$, because
\begin{align*}
h_\varphi(x,y)
&= h(\varphi^{-1}(x), \varphi^{-1}(y)) \\
&= h(\varphi^{-1}(\ell_p^{-1})_*(x), \varphi^{-1}(\ell_p^{-1})_*(y))\\
&= (\ell_p)^*h_\varphi(x,y).
\end{align*}
Further $[h_\varphi]=[\varphi^{-1}]^t[h][\varphi^{-1}]$.
Therefore $\aut(\frakg)$ acts on the space $\frakM(G)$ of all left invariant metrics on
$G$ by the rule $(\varphi,h)\mapsto h_\varphi$ or equivalently
$([\varphi],[h])\mapsto [h_\varphi]$.
If $G$ is simply connected, it is well known that
there is an isomorphism $\aut(G)\to\aut(\frakg)$ by $\theta\mapsto \theta_*$.
In this case, thus two actions of $\aut(G)$ and $\aut(\frakg)$ on $\frakM(G)$ are the same.
That is, we have a commutative diagram
$$
\CD
\aut(G) @.\,\,\,\, \times \,\,@. \frakM(G) @>>> \frakM(G),
@. \quad (\theta,h)@.\mapsto @.h_\theta \\
@VVV @. @VV{}V @VV{}V \quad @VVV @. @VV{=}V\\
\aut(\frakg) @. \,\,\,\,\times\,\,\,\, @. \frakM(G) @>>> \frakM(G),
@.\quad\, (\theta_*,h)@.\mapsto @. h_{\theta_*}
\endCD
$$

A left invariant Lorentzian metric $h'$ on $G$ is
{\it equivalent up to automorphism} to a left invariant Lorentzian metric $h$
if there exists $\varphi\in\aut(\frakg)$ such that
$[h']=[\varphi]^t[h][\varphi]$.
In this case we write $h'\sim h$ or $[h']\sim [h]$,
and we often say that $[h']$ is equivalent up to automorphism to $[h]$
or $[h']$ is {\it cogredient} to $[h]$ by $[\varphi]$.

Now we shall confine ourselves to the case of dimension three.
Let $h$ be a left invariant Lorentzian metric on $G$.
Then it induces a Lorentzian inner product on the vector space $\frakg$.
We denote this inner product by $h$.
We choose an orthonormal basis $\frakB_0$ on $\frakg$ for $h$.
Then
$$
[h]_{\frakB_0}=J_{2,1}
=\left[\begin{matrix}1&0&\hspace{8pt}0\\0&1&\hspace{8pt}0\\0&0&-1\end{matrix}\right].
$$
For any basis $\frakB$ of $\frakg$,
if $P$ is the transition matrix from the basis $\frakB$ to the basis $\frakB_0$,
then
$$
[h]_\frakB=P^t\, [h]_{\frakB_0}P.
$$
This implies that
\begin{itemize}
\item $\det([h]_\frakB)<0$ and
\item $[h]_\frakB$ cannot be a diagonal matrix with negative diagonals.
\end{itemize}

\subsection{The three-dimensional non-unimodular Lie algebras}\label{sec: Lie algebras}

There are uncountably many nonisomorphic three-dimensional non-unimodular, solvable Lie
algebras and a basis may be chosen so that
\begin{enumerate}
\item[{(a)}] $[x,y]=0,\,\, [z,x]=x,\,\, [z,y]=y$, or %
\end{enumerate}
\begin{enumerate}
\item[{(b)}] $[x,y]=0,\,\, [z,x]=y,\,\, [z,y]=-cx+2y$
\end{enumerate}
where $c\in\bbr$. Note that
$\ad(z)=\left[\begin{matrix}0&-c\\1&\hspace{8pt}2\end{matrix}\right]$ has
trace $2$ and determinant $c$. A reference is \cite{HL2009_MN}.
Such a Lie algebra is isomorphic to either
$\frakg_I$ or $\frakg_c$ for some $c\in\bbr$ where
\begin{align*}
\frakg_I & \cong \bbr^2\rtimes_{\sigma_{_I}}\bbr, \text{ where $\sigma_{_I}(t)= \left[ \begin{matrix} t & 0 \\
0 & {t} \end{matrix} \right]$;}\\
\frakg_{c} & \cong  \bbr^2\rtimes_{\sigma_{_c}}\bbr,
\text{ where $\sigma_{_c}(t)= \left[ \begin{matrix} 0 & -ct \\
t & \hspace{8pt}{2t} \end{matrix} \right]$}
\end{align*}
with a ``natural" basis
$$
x=\left(\left[ \begin{matrix} 1 \\ 0 \end{matrix}\right],0\right), \,
y=\left(\left[ \begin{matrix} 0 \\ 1 \end{matrix} \right],0\right),\,
z=\left(\left[ \begin{matrix} 0 \\ 0 \end{matrix} \right],1\right).
$$
Then they satisfy Lie bracket conditions (a) or (b), respectively.


For any non-unimodular solvable Lie algebra $\frakg$, its group of automorphisms is given as follows:
\begin{itemize}
    \item [(1)] The Lie group $\aut(\frakg_I)$ is isomorphic to
\begin{align*}
\left\{\left[
\begin{matrix} \mathrm{GL}(2,\bbr)& *\\ 0 & 1\end{matrix}\right]
\bigg|\,\, *\in\bbr^2\right\}.
\end{align*}
    \item [(2)] For each $c\in\bbr$, the Lie group $\aut(\frakg_c)$ is isomorphic to
\begin{align*}
\left\{\left[
\begin{matrix} \beta-\alpha & -c\alpha & *\\
\alpha & \beta+\alpha & *\\ 0 & 0 & 1\end{matrix}\right]
\bigg|
\begin{array}{l}
\alpha,\beta,*\in\bbr,\\\beta^2 + (c-1)\alpha^2 \ne 0
\end{array}\right\}.
\end{align*}
\end{itemize}

The three-dimensional non-unimodular Lie algebra $\frakg_I$ or
$\frakg_c$ is the Lie algebra of the connected and
simply connected three-dimensional Lie group
\begin{align*}
G_I & \cong \bbr^2\rtimes_{\varphi_{_I}}\bbr, \text{ where $\varphi_{_I}(t)= \left[ \begin{matrix} e^t & 0 \\
0 & {e^t} \end{matrix} \right]$, or}\\
G_c & \cong  \bbr^2\rtimes_{\varphi_{_c}}\bbr, \text{ where}
\end{align*}
{\footnotesize
$$
\varphi_{_c}(t) = \begin{cases} e^t\frac{e^{wt}+e^{-wt}}{2}\left[
\begin{matrix} 1& 0 \\ 0 & 1
\end{matrix} \right] +e^t\frac{e^{wt}-e^{-wt}}{2w}\left[
\begin{matrix} -1 & -c \\ 1& 1
\end{matrix} \right] & \text{if $w=\sqrt{1-c}\ne 0$}, \\
e^t\left[
\begin{matrix} 1& 0 \\ 0 & 1
\end{matrix} \right] +e^t t\left[
\begin{matrix} -1 & -1 \\ 1& 1
\end{matrix} \right]& \text{if $c = 1$}.
\end{cases}
$$
}

\section{Left invariant Lorentzian metrics and curvatures}\label{curvature-Lorentzian}

Let $G=G_I$ or $G_c$ and let 
$h\in\frakM(G)$ 
of signature $(+,+,-)$.
In what follows, we shall classify those $h$ up to automorphism and then
we shall compute the following associated curvatures
\begin{enumerate}
\item Ricci operator and curvature,
\item sectional curvature,
\item scalar curvature.
\end{enumerate}
Let $\varphi\in\aut(\frakg)$.
Let $h'=h_\varphi$ and let $\nabla'$ be the Levi-Civita connection associated to $(G,h')$.
Then for all $u,v,w\in\frakg$, we have
\begin{align*}
2h'(\nabla'_{\varphi(u)}\varphi(v),\varphi(w))
&=h'([\varphi(u),\varphi(v)],\varphi(w))+h'([\varphi(w),\varphi(u)],\varphi(v))\\
&\hspace{1cm}+h'([\varphi(w),\varphi(v)],\varphi(u))\\
&=h'(\varphi([u,v]),\varphi(w))+h'(\varphi([w,u]),\varphi(v))\\
&\hspace{1cm}+h'(\varphi([w,v]),\varphi(u))\\
&=h([u,v],w)+h([w,u],v)+h([w,v],u)\\
&=2h(\nabla_uv,w)\\
&=2h'(\varphi(\nabla_uv),\varphi(w))
\end{align*}
This reduces to $\nabla'_{\varphi(u)}\varphi(v)=\varphi(\nabla_uv)$, or the following diagram is commutative
$$
\CD
\frakg\x\frakg@>{\nabla}>>\frakg\\
@V{\varphi\x\varphi}VV@VV{\varphi}V\\
\frakg\x\frakg@>{\nabla'}>>\frakg
\endCD
$$
Therefore, the classification of the left invariant Lorentzian metrics up to automorphism leads to
the study of the left invariant Lorentzian metrics which leave all the curvature properties invariant.

We will begin with some necessary definitions.
Let $\nabla:\frakg\x\frakg\to\frakg$ be the Levi-Civita connection associated to $(G,h)$.
This is characterized by Koszul formula
$$
2h(\nabla_uv,w)=h([u,v],w)+h([w,u],v)+h([w,v],u).
$$
The \emph{Riemann curvature tensor} of $(G,h)$ is defined associated to each $u,v\in\frakg$ to be the linear transformation
$$
R_{uv}=\nabla_{[u,v]}-[\nabla_u,\nabla_v].
$$
The \emph{Ricci curvature} $\ric$ is the {symmetric} tensor defined, for any $u,v\in\frakg$, as the trace of the linear transformation $w\longmapsto R_{uw}v$.
Notice that The Riemann curvature tensor is completely determined by the Ricci curvature $\ric$.
The \emph{Ricci operator} $\Ric:\frakg\to\frakg$ is given the relation $h(\Ric(u),v)=\ric(u,v)$.
Because of the symmetries of the Ricci curvature, the Ricci operator $\Ric$ is $h$-self adjoint.
The \emph{scalar curvature} $\rho$ is the trace of the Ricci operator.
If $u,v\in\frakg$ are linearly independent, the number
$$
\kappa(u,v)=\frac{h(R_{uv}u,v)}{h(u,u)h(v,v)-h(u,v)^2}
$$
is called the \emph{sectional curvature} associated with $u,v$.

For an explicit computation, we fix an orthonormal basis $\frakB=\{y_i\}$ for $\frakg$ with respect to $h$.
So, $h(y_1,y_1)=h(y_2,y_2)=1, h(y_3,y_3)=-1$ and $h(y_i,y_j)=0$ for $i\ne j$.
From definitions, we immediately obtain
\begin{align*}
\kappa(y_1,y_2)&=\hspace{8pt}h(R_{y_1y_2}y_1,y_2),\\
\kappa(y_1,y_3)&=-h(R_{y_1y_3}y_1,y_3),\\
\kappa(y_2,y_3)&=-h(R_{y_2y_3}y_2,y_3),
\end{align*}
and
\begin{align*}
\ric(u,v)
&= h(R_{uy_1}v,y_1)+h(R_{uy_2}v,y_2)-h(R_{uy_3}v,y_3).
\end{align*}

\begin{Rmk}
We recall that the Ricci curvature $r(u)$ and the Ricci transformation $\hat{r}:\frakg\to\frakg$
are given in the Riemannian case as follows
\begin{align*}
r(u)&= g(R_{uy_1}u,y_1)+g(R_{uy_2}u,y_2)+g(R_{uy_3}u,y_3),\\
&=g\left(\sum_i R_{y_iu}y_i,u\right),\\
\hat{r}(u)&=\sum_i R_{y_iu}y_i.
\end{align*}
From the definition $h(\Ric(u),v)=\ric(u,v)$, we remark that
$\ric(u,u)$ and $\Ric$ are the Lorentzian versions of $r(u)$ and $\hat{r}$, respectively.
\end{Rmk}

Let
\begin{align*}
\Big[\Ric\Big]_{\{y_i\}}
= \Big[R_{ij}\Big].
\end{align*}
Then
\begin{align*}
R_{ij}&=\hspace{8.5pt}h(\Ric(y_j),y_i)=\hspace{8.5pt}\ric(y_j,y_i)=\hspace{8.5pt}\ric(y_i,y_j)=\hspace{8.5pt}R_{ji}\text{ for $i,j=1,2$},\\
R_{3j}&=-h(\Ric(y_j),y_3)=-\ric(y_j,y_3)=-\ric(y_3,y_j) = -R_{j3} \text{ for $j=1,2$}.
\end{align*}
This means that the matrix of the $h$-self adjoint operator $\Ric$
$$
\Big[\Ric\Big]=\left[\begin{matrix}\hspace{8pt}R_{11}&\hspace{8pt}R_{12}&R_{13}\\
\hspace{8pt}R_{12}&\hspace{8pt}R_{22}&R_{23}\\-R_{13}&-R_{23}&R_{33}
\end{matrix}\right]
$$
is Lorentzian symmetric.

\begin{Prop}[{\cite[Ex.~\!19 (p.~\!261)]{ONeill}}]
Let $V$ be a real vector space with Lorentzian inner product $h$ of signature $(2,1)$.
Every $h$-self adjoint operator $T$ on $V$
has a matrix of exactly one of the following four types:

\noindent
Relative to an $h$-orthonormal basis,
$$
\left[\begin{matrix}a&0&0\\0&b&0\\0&0&c\end{matrix}\right],\quad\text{ or }\quad
\left[\begin{matrix}a&\hspace{8pt}0&0\\0&\hspace{8pt}\alpha&\beta\\0&-\beta&\alpha\end{matrix}\right](\beta\ne0);
$$
\noindent
Relative to a basis $\{e,u,v\}$ with non-trivial products $h(u,v)=1=h(e,e)$,
$$
\left[\begin{matrix}a&0&0\\0&b&\epsilon\\0&0&b\end{matrix}\right] (\epsilon=\pm1),\quad\text{ or }\quad
\left[\begin{matrix}a&0&1\\1&a&0\\0&0&a\end{matrix}\right].
$$
\end{Prop}

The vectors
$$
e_1=e,\ e_2=\tfrac{1}{\sqrt{2}}(u+v),\ e_3=\tfrac{1}{\sqrt{2}}(u-v)
$$
form an $h$-orthonormal basis with respect to which the last two matrices become respectively
$$
\left[\begin{matrix}a&0&0\\0&b+\frac{\epsilon}{2}&-\frac{\epsilon}{2}\\ 0&\frac{\epsilon}{2}&b-\frac{\epsilon}{2}\end{matrix}\right]  (\epsilon=\pm1),\quad\text{ or }\quad
\left[\begin{matrix}a&\frac{1}{\sqrt{2}}&-\frac{1}{\sqrt{2}}\\\frac{1}{\sqrt{2}}&a&\hspace{8pt}0\\ \frac{1}{\sqrt{2}}&0&\hspace{8pt}a\end{matrix}\right].
$$
These four types are called \emph{O'Neill type} or \emph{Segr\'{e} type} $\{11,1\}, \{1z\bar{z}\}, \{21\}$ and $\{3\}$, respectively.
Note that the O'Neill type $\{21\}$ is of the form
$$
\left[\begin{matrix}a&0&0\\0&b\pm 1 &\mp1\\0&\pm1&b\mp1\end{matrix}\right] \sim
\left[\begin{matrix}a&0&0\\0&b\pm 1 & 1\\0&-1&b\mp1\end{matrix}\right],
$$
because
$$
\left[\begin{matrix}a&0&0\\0&b\pm\frac{1}{2}&\mp\frac{1}{2}\\ 0&\pm\frac{1}{2}&b\mp\frac{1}{2}\end{matrix}\right] \sim
\left[\begin{matrix}a&0&0\\0&b\pm 1 &\mp1\\0&\pm1&b\mp1\end{matrix}\right]
$$
via {\footnotesize $\left[\begin{matrix}1&0&0\\0&\cosh{t}&\sinh{t}\\0&\sinh{t}&\cosh{t}\end{matrix}\right]$}$\in\O(2,1)$
where $t=-\ln\sqrt{2}$. Notice further that
{\footnotesize$\left[\begin{matrix}a&0&0\\0&b+1&1\\0&-1&b-1\end{matrix}\right] \nsim
\left[\begin{matrix}a&0&0\\0&b-1&1\\0&-1&b+1\end{matrix}\right]$}.
The O'Neill type $\{3\}$ is of the form
$$
\left[\begin{matrix}a&1&-1\\1&a&\hspace{8pt}0\\1&0&\hspace{8pt}a\end{matrix}\right]
=\left[\begin{matrix}1&\hspace{8pt}0&\hspace{8pt}0\\0&\hspace{8pt}\frac{3}{2\sqrt{2}}&-\frac{1}{2\sqrt{2}}\\
0&-\frac{1}{2\sqrt{2}}&\hspace{8pt}\frac{3}{2\sqrt{2}}\end{matrix}\right]^{-1}
\left[\begin{matrix}a&-\frac{1}{\sqrt{2}}&\frac{1}{\sqrt{2}}\\\frac{1}{\sqrt{2}}&\hspace{8pt}a&0\\ \frac{1}{\sqrt{2}}&\hspace{8pt}0&a\end{matrix}\right] \left[\begin{matrix}1&\hspace{8pt}0&\hspace{8pt}0\\0&\hspace{8pt}\frac{3}{2\sqrt{2}}&-\frac{1}{2\sqrt{2}}\\
0&-\frac{1}{2\sqrt{2}}&\hspace{8pt}\frac{3}{2\sqrt{2}}\end{matrix}\right].
$$

Thus we can always construct an orthonormal basis $\frakB$ for $\frakg$ such that
$[\Ric]_{\frakB}$ takes exactly one of the four O'Neill types.

The O'Neill type $\{11,1\}$ (the comma is used to separate the spacelike and timelike
eigenvectors) denotes a diagonalizable self adjoint operator.
The O'Neill type $\{1z\bar{z}\}$ denotes a self adjoint operator with one real and two complex conjugate eigenvalues.
A self adjoint operator of O'Neill type $\{21\}$ has two eigenvalues (one of which has multiplicity two),
each associated to a one-dimensional eigenspace.
A self adjoint operator of O'Neill type $\{3\}$ has three equal eigenvalues associated to a one-dimensional eigenspace.
The eigenvalues of $[\Ric]$ are called the \emph{Ricci eigenvalues}.
Hence, the O'Neill type of $\Ric$ is determined by its eigenvalues with associated eigenspaces.

\begin{Rmk}
In the Riemannian case, since $[\hat{r}]$ is symmetric, the eigenvalues of $[\hat{r}]$ are real, called the \emph{principal Ricci curvatures},
and the signature of the symmetric matrix $[\hat{r}]$ is well-defined, called the \emph{signature} of the Ricci curvature.
However, in the Lorentzian case, the Lorentzian version $[\Ric]$ of $\hat{r}$ may have complex eigenvalues.
In fact, complex eigenvalues occur only in the following cases:
\begin{itemize}
  \item $G_c$ with $c>1$, (2) in Theorem~\ref{thm: c>1-curvature};
  \item $G_1$, (3),(5) in Theorem~\ref{thm: c=1-curvature};
  \item $G_c$ with $c<1$, (10-1), (11) in Theorem~\ref{thm: c<1-curvature}.
\end{itemize}
As a result, the principal Ricci curvatures are defined always and the signature of the Ricci curvature
is defined whenever the eigenvalues are real numbers.
\end{Rmk}

The following two lemmas will be used to determine the O'Neill type of the Ricci operator.

\begin{Lemma}\label{lemma:OT1}
If  $[\Ric]_{\frakB}$ is of the form
$$
\left[\begin{matrix}a&\hspace{6pt}0&0\\0&\hspace{6pt}b&r\\0&-r&d\end{matrix}\right]\ \text{ or }\
\left[\begin{matrix}\hspace{6pt}b & 0 & r \\ \hspace{6pt}0 & a&0\\-r&0&d\end{matrix}\right]
$$
for some
orthonormal basis $\frakB$ for $\frakg$ with respect to the left invariant Lorentzian metric $h$ on $G_c$,
then $\Ric$ is of
$$
\begin{array}{lll}
&\text{{O'Neill type} $\{11,1\}$}\quad &\text{if $r=0$ or $($$r\ne 0$ and $D>0$$)$};\\
&\text{{O'Neill type} $\{1z\bar{z}\}$}\quad  &\text{if $r\ne0$ and $D<0$};\\
&\text{{O'Neill type} $\{21\}$}\quad  &\text{if $r\ne0$ and $D=0$,}
\end{array}
$$
where $D=(b-d)^2-4r^2$.
\end{Lemma}

\begin{proof} We may assume that {\footnotesize$[\Ric]_{\frakB}=\left[\begin{matrix}a&\hspace{8pt}0&0\\0&\hspace{8pt}b&r\\0&-r&d\end{matrix}\right]$} and $r\ge 0$ because
\begin{align*}
\left[\begin{matrix}\hspace{8pt}b & 0 & r \\ \hspace{8pt}0 & a&0\\-r&0&d\end{matrix}\right]
\sim \left[\begin{matrix}a&\hspace{8pt}0&0\\0&\hspace{8pt}b&r\\0&-r&d\end{matrix}\right], \quad
\left[\begin{matrix}a&\hspace{8pt}0&0\\0&\hspace{8pt}b&r\\0&-r&d\end{matrix}\right]
\sim \left[\begin{matrix}a&0&\hspace{8pt}0\\0&b&-r\\0&r&\hspace{8pt}d\end{matrix}\right]
\end{align*}
via {\footnotesize $\left[\begin{matrix}0 & 1 & 0 \\ 1 & 0&0\\0&0&1\end{matrix}\right], \left[\begin{matrix}1 & \hspace{8pt}0 & 0 \\ 0 & -1&0\\0&\hspace{8pt}0&1\end{matrix}\right]$}$\in O(2,1)$, respectively.
When $r=0$ or ($r\ne0$ and $b=d$), clearly the results hold.
Now assume that $r>0$ and $b\ne d$.
Since the Ricci operator $\Ric$ is $h$-self adjoint,
the O'Neill type of $\Ric$ is determined by its eigenvalues with associated eigenspaces.
The eigenvalues of $[\Ric]_{\frakB}$ are $a, \frac{b+d}{2}\pm \frac{\sqrt{D}}{2}$ where $D=(b-d)^2-4r^2$.

If $D<0$, the only possibility is the O'Neill type $\{1z\bar{z}\}$. In fact, we can see that
if we choose an orthonormal basis $\frakB'$ given by
$$
[\id]_{\frakB',\frakB} =
\left[\begin{matrix} 1&0&0\\0&\cosh\theta&\sinh\theta\\
0&\sinh\theta&\cosh\theta\end{matrix}\right]\in O(2,1)
$$
where $\theta=\tanh^{-1}\Big(\frac{d-b}{2r+\sqrt{-D}}\Big)$, then
$$
[\Ric]_{\frakB'} = \left[\begin{matrix}a&0&0\\0&\frac{b+d}{2}&\frac{\sqrt{-D}}{2}\\0&-\frac{\sqrt{-D}}{2}&\frac{b+d}{2}\end{matrix}\right],
$$
which is of O'Neill type $\{1z\bar{z}\}$.

Now suppose $D=0$. Then $[\Ric]_{\frakB}$  has two eigenvalues $a,\frac{b+d}{2}$ ($\frac{b+d}{2}$ has multiplicity two), each associated to a one-dimensional eigenspace.
Hence $\Ric$ is of O'Neill type $\{21\}$.
Since $D=0$, $b-d=\pm 2r$.
If we choose an orthonormal basis $\frakB'$ given by
$$
[\id]_{\frakB',\frakB} =
\left[\begin{matrix} 1&0&0\\0&\cosh\theta&\sinh\theta\\
0&\sinh\theta&\cosh\theta\end{matrix}\right]\in O(2,1)
$$
where  $\theta=\mp\ln \sqrt{r}$ as $b-d=\pm 2r$, respectively,
then
\begin{align*}
[\Ric]_{\frakB'} &= \left[\begin{matrix}a&\hspace{8pt}0&0\\0&\frac{b+d}{2}\pm1&1\\0&-1&\frac{b+d}{2}\mp1\end{matrix}\right]
\end{align*}
which is of O'Neill type $\{21\}$.

Finally assume $D>0$. Then $[\Ric]_{\frakB}$ has  real eigenvalues $a, \frac{b+d}{2}\pm \frac{\sqrt{D}}{2}$ and so diagonalizable.  Hence $\Ric$ is of O'Neill type $\{11,1\}$.
Since $4r^2=(d-b+\sqrt{D})(d-b-\sqrt{D})$ and $D>0$, $d-b+\sqrt{D}\ne\pm2r$.
If we choose an orthonormal basis $\frakB'$ given by
$$
[\id]_{\frakB',\frakB} =
\left[\begin{matrix} 1&0&0\\0&\cosh\theta&\sinh\theta\\
0&\sinh\theta&\cosh\theta\end{matrix}\right]\in O(2,1)
$$
where
$$
\theta=\left\{\begin{array}{ll}
\tanh^{-1}\Big(\frac{d-b+\sqrt{D}}{2r}\Big), &\Big|\frac{d-b+\sqrt{D}}{2r}\Big|<1;\\
\tanh^{-1}\Big(\frac{2r}{d-b+\sqrt{D}}\Big), &\Big|\frac{d-b+\sqrt{D}}{2r}\Big|>1,
\end{array} \right.
$$
 then
\begin{align*}
[\Ric]_{\frakB'} &= \left\{\begin{array}{ll}
\left[\begin{matrix}a&0&0\\0&\frac{b+d+\sqrt{D}}{2}&0\\0&0&\frac{b+d-\sqrt{D}}{2}\end{matrix}\right], &\Big|\frac{d-b+\sqrt{D}}{2r}\Big|<1;\\
\left[\begin{matrix}a&0&0\\0&\frac{b+d-\sqrt{D}}{2}&0\\0&0&\frac{b+d+\sqrt{D}}{2}\end{matrix}\right], &\Big|\frac{d-b+\sqrt{D}}{2r}\Big|>1
\end{array} \right.
\end{align*}
which is of O'Neill type $\{11,1\}$.
\end{proof}

\begin{Lemma}\label{lemma:OT2}
If  $[\Ric]_{\frakB}$ is of the form {$\left[\begin{matrix}b&c&0\\c&d&0\\0&0&a\end{matrix}\right]$} for some
orthonormal basis $\frakB$ for $\frakg$ with respect to the left invariant Lorentzian metric $h$ on $G_c$,
then $\Ric$ is of O'Neill type $\{11,1\}$.
\end{Lemma}

\begin{proof} Since {\footnotesize$\left[\begin{matrix}b&c\\c&d\end{matrix}\right]$} is symmetric, it is orthogonally diagonalizable. Thus there is $B\in O(2) \subset O(2,1)$ such that
$$
\left[\begin{matrix} B&0\\0&\det(B)\end{matrix}\right]^{-1}
\left[\begin{matrix}b&c&0\\c&d&0\\0&0&a\end{matrix}\right]
\left[\begin{matrix} B&0\\0&\det(B)\end{matrix}\right]
=\left[\begin{matrix}\frac{b+d+\sqrt{D'}}{2}&0&0\\0&\frac{b+d-\sqrt{D'}}{2}&0\\0&0&a\end{matrix}\right]
$$
where $D'=(b-d)^2+4c^2$. Hence $\Ric$ is of O'Neill type $\{11,1\}$.
\end{proof}

Note also that the matrix of $\ric$ with respect to the orthonormal basis $\{y_i\}$ is
\begin{align}\label{Ric:ric}
\Big[\ric\Big]_{\{y_i\}}
=J_{2,1}\Big[\Ric\Big]_{\{y_i\}}
=\left[\begin{matrix}R_{11}&R_{12}&\hspace{8pt}R_{13}\\
R_{12}&R_{22}&\hspace{8pt}R_{23}\\R_{13}&R_{23}&-R_{33}
\end{matrix}\right]
\end{align}
and the scalar curvature $\rho$ is
\begin{equation}
\label{eq:scalar}
\begin{aligned}
\rho&=\mathrm{trace}(\Ric)=R_{11}+R_{22}+R_{33}\\
&=\ric(y_1,y_1)+\ric(y_2,y_2)-\ric(y_3,y_3),\\
&=2\left(\kappa(y_1,y_2)+\kappa(y_1,y_3)+\kappa(y_2,y_3)\right).
\end{aligned}
\end{equation}

Clearly,  the Ricci eigenvalues of curvature operator $\Ric$ are independent of the choice of
basis for $\frakg$. Now show that the signature of the Ricci curvature tensor $\ric$ is independent of the choice of basis for $\frakg$. Let $\{x_1,x_2,x_3\}$ be another, not necessarily orthonormal, basis for $\frakg$.
By writing
$x_j=\sum_i a_{ij}y_i$ for some $a_{ij}$,
we have
\begin{align*}
\ric(x_i,x_j)&=\ric(\sum_k a_{ki}y_k, \sum_\ell a_{\ell j}y_\ell)\\
&=\sum_{k,\ell}a_{ki}\cdot \ric(y_k,y_\ell)\cdot a_{\ell j}
\end{align*}
Hence
\begin{align}\label{id:ric}
\Big[\ric\Big]_{\{x_i\}}=A^t\cdot \Big[\ric\Big]_{\{y_i\}}\cdot A
\end{align}
where $A=(a_{ij})$ is the associated transition matrix $[\id]_{\{x_i\},\{y_i\}}$.
This implies that $[\ric]_{\{x_i\}}$ and $[\ric]_{\{y_i\}}$ have the determinants of same sign,
hence the products of their eigenvalues have the same sign.
Moreover, we obtain
\begin{align*}
\Big[\Ric\Big]_{\{x_i\}}
= A^{-1} \Big[\Ric\Big]_{\{y_i\}} A
= A^{-1} J_{2,1}\Big[\ric\Big]_{\{y_i\}} A
= A^{-1} J_{2,1}(A^{-1})^{t}\Big[\ric\Big]_{\{x_i\}}
\end{align*}

\begin{Lemma}
$[\ric]_{\{x_i\}}$ and $[\ric]_{\{y_i\}}$ have the same signature.
Consequently, the signature of $[\ric]$ is defined.
\end{Lemma}

\begin{proof}
Suppose $[\ric]_{\{y_i\}}$ has the signature $(0,0,0)$. Since $[\ric]_{\{y_i\}}=O$ is symmetric,  $[\ric]_{\{x_i\}}=A^t[\ric]_{\{y_i\}}A=O$.
Hence $[\ric]_{\{x_i\}}$ and $[\ric]_{\{y_i\}}$ have the same signature $(0,0,0)$.
Now assume that the signature of $[\ric]_{\{y_i\}}$ is not $(0,0,0)$.
Since $A$ is invertible and $[\ric]_{\{x_i\}}=A^t [\ric]_{\{y_i\}} A$,  $[\ric]_{\{x_i\}}$ is positive(negative) definite
if and only if $[\ric]_{\{y_i\}}$ is positive(negative) definite
if and only if all the eigenvalues of $[\ric]_{\{y_i\}}$ are positive(negative).
It follows that the result holds when the signatures are $(+,+,+)$, $(-,-,-)$, $(+,+,-)$, $(+,0,-)$ and  $(+,-,-)$.

If $\lambda=0$ is an eigenvalue of $[\ric]_{\{y_i\}}$, then $[\ric]_{\{y_i\}}\bfx=\mathbf{0}$ for some $\bfx\ne\mathbf{0}$.
Thus $A^{-1}\bfx\ne\mathbf{0}$ and
$$
[\ric]_{\{x_i\}}(A^{-1}\bfx) = (A^t [\ric]_{\{y_i\}} A)(A^{-1}\bfx) = A^t [\ric]_{\{y_i\}}\bfx=\mathbf{0}.
$$
It follows that $0$ is an eigenvalue of $[\ric]_{\{x_i\}}$.
If $\bfx$ and $\bfy$ are linearly independent eigenvectors of $[\ric]_{\{y_i\}}$ corresponding to the eigenvalue $0$,
then  $A^{-1}\bfx$ and $A^{-1}\bfy$ are linearly independent eigenvectors of $[\ric]_{\{x_i\}}$ corresponding to the eigenvalue $0$. Based on this fact,
the result holds when the signatures are $(+,+,0)$, $(+,0,0)$, $(0,0,0)$, $(0,0,-)$ and  $(0,-,-)$.
\end{proof}

Let $g$ be a left invariant Riemannian metric on a connected Lie group $G$ of dimension three.
Let $\rho$ denote the sectional curvature,
and let $r$ denote the Ricci curvature of $(G,g)$.
For any orthogonal vectors $u$ and $v$, the sectional curvature $\kappa(u,v)$
associated with $u$ and $v$ is given by the formula (\cite[p.~\!306]{Milnor})
$$
\kappa(u,v)=||u\x v||^2\,\frac{\rho}{2}-r(u\x v).
$$

Now, we give its Lorentzian version.
Let $h$ be a left-invariant Lorentzian inner product on a Lie algebra $\frakg$ of dimension three.
We fix an orthonormal basis $\{y_1,y_2,y_3\}$ with respect to $h$.
The Lorentzian cross product is defined as follows
$$
y_1\x y_2=-y_3,\ y_2\x y_3=y_1,\ y_3\x y_1=y_2.
$$

Let $u$ and $v$ be orthogonal vectors, i.e. $h(u,v)=0$.
Write
$$
u=\sum_iu_iy_i,\quad
v=\sum_iv_iy_i.
$$
Since the cross product is bilinear and skew-symmetric, we see that
\begin{align*}
&u\x v=\sum_i w_iy_i,\\
&||u\x v||^2=w_1^2+w_2^2-w_3^2,
\end{align*}
where $w_1=u_2v_3-u_3v_2,\ w_2=u_3v_1-u_1v_3,\ w_3=-(u_1v_2-u_2v_1)$.

Further, we have
\begin{align*}
\ric(u\x v,&u\x v)
=\ric\left(\sum_i w_iy_i,\sum_i w_iy_i\right)\\
&=\sum_{i,j} w_iw_j \left(h(R_{y_iy_1}y_j,y_1)+h(R_{y_iy_2}y_j,y_2)-h(R_{y_iy_3}y_j,y_3)\right)\\
&=w_1^2\left(h(R_{y_1y_2}y_1,y_2)-h(R_{y_1y_3}y_1,y_3)\right)\\
&\quad+w_2^2\left(h(R_{y_1y_2}y_1,y_2)-h(R_{y_2y_3}y_2,y_3)\right)\\
&\quad+w_3^2\left(h(R_{y_1y_3}y_1,y_3)+h(R_{y_2y_3}y_2,y_3)\right)\\
&\quad-2w_1w_2\ h(R_{y_1y_3}y_2,y_3)-2w_1w_3\ h(R_{y_1y_2}y_2,y_3)\\
&\quad+2w_2w_3\ h(R_{y_1y_2}y_1,y_3).
\end{align*}
Hence, by substituting with (\ref{eq:scalar}), we get
\begin{align*}
&||u\x v||^2\ \frac{\rho}{2}-\ric(u\x v,u\x v)\\
&=-w_3^2\ h(R_{y_1y_2}y_1,y_2)-w_2^2\ h(R_{y_1y_3}y_1,y_3)-w_1^2\ h(R_{y_2y_3}y_2,y_3)\\
&\ +2w_1w_2\ h(R_{y_1y_3}y_2,y_3)+2w_1w_3\ h(R_{y_1y_2}y_2,y_3)-2w_2w_3\ h(R_{y_1y_2}y_1,y_3).
\end{align*}
On the other hand, we have
\begin{align*}
h(R_{uv}u,v)
&=\sum_{i,j,k,\ell}u_iv_ju_kv_\ell\ h(R_{y_iy_j}y_k,y_\ell)\\
&=\sum_{i<j, k<\ell}(u_iv_j-u_jv_i)(u_kv_\ell-u_\ell v_k)h(R_{y_iy_j}y_k,y_\ell)\\
&=-\left(||u\x v||^2\ \frac{\rho}{2}-\ric(u\x v,u\x v)\right).
\end{align*}
In conclusion, we have the following result.

\begin{Lemma}\label{lamma:Lorentzian version of Milnor's formula}
Let $h$ be a left invariant Lorentzian metric on a connected Lie group $G$ of dimension three.
If $u$ and $v$ are orthogonal vectors, then
\begin{align}\label{identity:sec}
h(u,u)h(v,v)\,\kappa(u,v)=-\left(||u\x v||^2\ \frac{\rho}{2}-\ric(u\x v,u\x v)\right).
\end{align}
In particular,
\begin{align*}
\kappa(y_1,y_2)&=\frac{\rho}{2}+\ric(y_3,y_3) = \frac{1}{2}(\hspace{8pt}R_{11}+R_{22}-R_{33}),\\
\kappa(y_2,y_3)&=\frac{\rho}{2}-\ric(y_1,y_1) = \frac{1}{2}(-R_{11}+R_{22}+R_{33}),\\
\kappa(y_3,y_1)&=\frac{\rho}{2}-\ric(y_2,y_2) = \frac{1}{2}(\hspace{8pt}R_{11}-R_{22}+R_{33}).
\end{align*}
\end{Lemma}

When the Ricci operator $\Ric$ is computed,
\eqref{Ric:ric} yields the Ricci curvature tensor $\ric$,
\eqref{eq:scalar} yields the scalar curvature $\rho$,
and Lemma~\ref{lamma:Lorentzian version of Milnor's formula} gives rise to the sectional curvatures $\kappa(y_i,y_j)$.
For this reason, only the Ricci operator will be described in the following sections.
The scalar curvature and section curvatures are included in Tables~\ref{tab1}, \ref{tab2}, \ref{tab3}, \ref{tab4} and \ref{tab5}.

\subsection{The case of $G_I$}
\label{sec:G_I}

In this section, we consider the three-dimensional non-unimodular Lie group $G_I$,
and then we will compute the moduli space $\aut(G_I)\bs\frakM(G_I)$ and the Ricci operator.

Recalling from Section~\ref{sec: Lie algebras} that
\begin{align*}
\aut(\frakg_I)= \left\{\left[
\begin{matrix} \mathrm{GL}(2,\bbr)& *\\ 0 & 1\end{matrix}\right]
\bigg|\,\, *\in\bbr^2\right\},
\end{align*}
we obtain the following result.

\begin{Thm}\label{thm:G_I}
Any left invariant Lorentzian metric on $G_I$ is equivalent up to
automorphism to a metric whose associated matrix is of the form
$$
\left[\begin{matrix} 1 & \hspace{8pt}0 & 0 \\ 0 & -1 & 0 \\ 0 & \hspace{8pt}0 & \mu \end{matrix} \right] \quad\text{or}\quad
\left[\begin{matrix} 1 & 0 & \hspace{8pt}0 \\ 0 & 1 & \hspace{8pt}0 \\ 0 & 0 & -\mu \end{matrix} \right] \quad\text{or}\quad
\left[\begin{matrix} 1 & 0 & 0 \\ 0 & 0 & 1 \\ 0 & 1 & 0 \end{matrix} \right]
$$
where $\mu> 0$.
\end{Thm}

\begin{proof}
Let $h$ be a left invariant Lorentzian metric on $G_I$.
Then $[h]=[h_{ij}]$ is a symmetric and non-degenerate real $3\times 3$ matrix.
Since $[h]$ is symmetric and {\footnotesize $\left[\begin{matrix} \mathrm{SO}(2) & \begin{matrix}0\\0\end{matrix} \\
\begin{matrix}0&0\end{matrix} & 1 \end{matrix} \right]$} $\subset\aut(\mathfrak{g})$,
we may assume that $h_{12}=0$.
First assume  $h_{11} h_{22}\ne 0$. Then $C =${\footnotesize
$\left[\begin{matrix}
1 & 0 & -\frac{h_{13}}{h_{11}}\\
0 & 1 & -\frac{h_{23}}{h_{22}} \\
0 & 0 & 1 \end{matrix}\right]$} $\in\aut(\frakg_I)$ and
 $C^t B^t[h]B C = \diag\{h_{11}, h_{22}, \mu\}$.
Since only one of $h_{11}, h_{22}$ and $\mu$ is negative,
either $h_{11} h_{22}<0, \mu>0$ or $h_{11}>0, h_{22}>0, \mu<0$.
If $h_{11} h_{22}<0$ and $\mu>0$, then we may assume that $h_{11}>0$ and $h_{22}<0$.
Let $D= \diag\{\frac{1}{\sqrt{h_{11}}}, \frac{1}{\sqrt{-h_{22}}}, 1\}\in\aut(\frakg_I)$.
Then
$D^tC^tB^t[h]BCD=\diag\{1,-1,\mu\}$.
If $h_{11}>0, h_{22}>0$ and $\mu<0$, then
$D=\diag\{\frac{1}{\sqrt{h_{11}}}, \frac{1}{\sqrt{h_{22}}}, 1\}\in\aut(\frakg_I)$
and $D^tC^tB^t[h]BCD=\diag\{1,1,\mu\}$.

Now assume $h_{11} h_{22}= 0$. Since one of $h_{11}$ and $h_{22}$ is non-zero, we may assume $h_{11}\ne0$ and $h_{22}=0$. Since $\det[h]=-h_{11}h_{23}^2<0$, we have $h_{11}>0$ and $h_{23}\ne0$.
Let $C =${\footnotesize
$\left[\begin{matrix}
\frac{1}{\sqrt{h_{11}}} & 0 & -\frac{h_{13}}{h_{11}}\\ 0 & 1 & 0 \\ 0 & 0 & 1 \end{matrix}\right]$} $\in\aut(\frakg_I)$. Then
 $C^tB^t[h]BC =$ {\footnotesize $\left[\begin{matrix} 1 & 0 & 0 \\ 0 & 0 & \mu \\0 & \mu & \nu \end{matrix} \right]$}. Since $\det(C^tB^t[h]BC)=-\mu^2<0$, we have $\mu\ne 0$.
Let $D =${\footnotesize
$\left[\begin{matrix}
1 & 0 & \hspace{8pt}0 \\ 0 & \frac{1}{\mu} & -\frac{\nu}{\mu}\\
0 & 0 & \hspace{8pt}1 \end{matrix}\right]$} $\in\aut(\frakg_I)$. Then
 $D^tC^tB^t[h]BCD =$ {\footnotesize $\left[\begin{matrix} 1 & 0 & 0 \\ 0 & 0 & 1 \\0 & 1 & 0 \end{matrix} \right]$}.

Finally it is easy to see that any three such distinct matrices are not equivalent.
\end{proof}

For each left invariant Lorentzian metric $h$ on $G_I$ whose associated matrix is given in Theorem~\ref{thm:G_I}, we obtain an orthonormal basis $\{y_i\}$ as follows:
\begin{itemize}
  \item [(1)] If $[h]=${\footnotesize$\left[\begin{matrix} 1 & \hspace{8pt}0 & 0 \\ 0 & -1 & 0 \\ 0 & \hspace{8pt}0 & \mu \end{matrix} \right]$} with $\mu> 0$,
\begin{align*}
  &y_1=x, \ y_2=\frac{1}{\sqrt{\mu}}z, \, y_3=y;\\
  &[y_1,y_2]=-\frac{1}{\sqrt{\mu}}y_1, \, [y_2,y_3]=\frac{1}{\sqrt{\mu}}y_3, \, [y_3,y_1]=0.
\end{align*}

  \item [(2)] If $[h]=${\footnotesize$\left[\begin{matrix} 1 & 0 & \hspace{8pt}0 \\ 0 & 1 & \hspace{8pt}0 \\ 0 & 0 & -\mu \end{matrix} \right]$} with $\mu> 0$,
 \begin{align*}
  &y_1=x, \, y_2=y, \, y_3=\frac{1}{\sqrt{\mu}}z;\\
  &[y_1,y_2]=0, \, [y_2,y_3]=-\frac{1}{\sqrt{\mu}}y_2, \,
  [y_3,y_1]=\frac{1}{\sqrt{\mu}}y_1.
\end{align*}

  \item [(3)] If $[h]=${\footnotesize$\left[\begin{matrix} 1 & 0 & 0 \\ 0 & 0 & 1 \\ 0 & 1 & 0 \end{matrix} \right]$},
  \begin{align*}
  &y_1=x, \, y_2=\frac{1}{\sqrt{2}}(y+z), \, y_3=\frac{1}{\sqrt{2}}(y-z);\\
  &[y_1,y_2]=-\frac{1}{\sqrt{2}}y_1, \,
  [y_2,y_3]=\frac{1}{\sqrt{2}}(y_2+y_3), \,
  [y_3,y_1]=-\frac{1}{\sqrt{2}}y_1.
 \end{align*}
\end{itemize}

With respect to the orthonormal basis $\{y_i\}$, we have the following:

\begin{Thm}\label{thm:G_I-curvature}
The Ricci operator $\Ric$ of the metric $h$ on $G_I$ is expressed as follows:
\begin{itemize}
  \item [(1)] If $[h]=\left[\begin{matrix} 1 & \hspace{8pt}0 & 0 \\ 0 & -1 & 0 \\ 0 & \hspace{8pt}0 & \mu \end{matrix} \right]$ with $\mu> 0$,
\begin{align*}
[\Ric]_{\{y_i\}}
&=\diag\Big\{-\frac{2}{\mu},-\frac{2}{\mu},-\frac{2}{\mu}\Big\}
\quad(\text{O'Neill type $\{11,1\}$})
\end{align*}

  \item [(2)] If $[h]=\left[\begin{matrix} 1 & 0 & 0 \\ 0 & 1 & 0 \\ 0 & 0 & -\mu \end{matrix} \right]$ with $\mu> 0$,
 \begin{align*}
[\Ric]_{\{y_i\}}
&=\diag\Big\{\frac{2}{\mu},\frac{2}{\mu},\frac{2}{\mu}\Big\}
\quad(\text{O'Neill type $\{11,1\}$})
\end{align*}

  \item [(3)] If $[h]=\left[\begin{matrix} 1 & 0 & 0 \\ 0 & 0 & 1 \\ 0 & 1 & 0 \end{matrix} \right]$,
 \begin{align*}
[\Ric]_{\{y_i\}}
&=\diag\{0,0,0\} \,\text{(flat)}
\quad(\text{O'Neill type $\{11,1\}$})
\end{align*}

\end{itemize}
\end{Thm}

\subsection{The case of $G_c$, $c>1$}\label{sec:c>1}

In this section, we consider the three-dimensional non-unimodular Lie group $G_c$ with $c>1$,
and then we will compute the moduli space $\aut(G_c)\bs\frakM(G_c)$ and the Ricci operator.

Recalling from Section~\ref{sec: Lie algebras} that
\begin{align*}
\aut(\frakg_c)= \left\{\left[
\begin{matrix} \beta-\alpha & -c\alpha & *\\
\alpha & \beta+\alpha & *\\ 0 & 0 & 1\end{matrix}\right]
\bigg|
\begin{array}{l}
\alpha,\beta,*\in\bbr,\\\beta^2 + (c-1)\alpha^2 \ne 0
\end{array}\right\},
\end{align*}
we obtain the following result.

\begin{Thm}\label{thm: c>1}
Any left invariant Lorentzian metric on $G_c$ with $c>1$ is equivalent up to
automorphism to a metric whose associated matrix is of the form
$$
\left[\begin{matrix}\mu & 0 & 0\\ 0 & 0 & 1\\ 0 & 1 & 0 \end{matrix}\right], \quad
\left[\begin{matrix} 1 & 1 & 0\\ 1 & \tau & 0\\ 0 & 0 & \mu \end{matrix}\right], \quad
\left[\begin{matrix} 1 & 1 & \hspace{8pt}0\\ 1 & \nu & \hspace{8pt}0\\ 0 & 0 & -\mu \end{matrix}\right]
$$
where $\mu>0$, $\tau<1$  and  $1<\nu\le c$.
\end{Thm}

\begin{proof}
Let $h$ be a left invariant Lorentzian metric on $G_c$ with $c>1$.
Then $[h]=[h_{ij}]$ is a symmetric and non-degenerate real $3\times 3$ matrix.

First suppose  $h_{11}h_{22}-h_{12}^2= 0$.
If $h_{12}\ne 0$, then $h_{11}\ne0, h_{22}\ne 0$.
Since $(h_{11}-h_{12})^2 + (c-1)h_{11}^2>0$, $B =${\footnotesize
$\left[\begin{matrix}
-h_{12} & -c\,h_{11} & 0\\
\hspace{8pt}h_{11} & 2h_{11}-h_{12} &0 \\
\hspace{8pt}0 & 0 & 1 \end{matrix}\right]$} $\in\aut(\frakg_c)$ and
$B^t [h] B = [h_{ij}']$ where $h_{11}'=h_{12}'=0$,  $h_{13}'\ne0$ and $h_{22}' \ne 0$.
Thus we may assume that $h_{12}=0$.
Then either $h_{11}\ne 0, h_{22}= 0$ or $h_{11}=0, h_{22}\ne 0$.
Suppose $h_{11}=0$ and $h_{22}\ne 0$. Since
$B =${\footnotesize
$\left[\begin{matrix}
-2 & -c & 0\\
\hspace{8pt}1 & \hspace{8pt}0 & 0 \\
\hspace{8pt}0 & \hspace{8pt}0 & 1 \end{matrix}\right]$} $\in\aut(\frakg_c)$ and
$B^t [h] B = [h_{ij}']$ where $h_{11}'=h_{22}\ne0$ and $h_{12}'=h_{22}'=0$,
we may assume that $h_{11}\ne 0$ and $h_{22}=0$.
Then we have $h_{23}\ne 0$,
$C =${\footnotesize
$\left[\begin{matrix}
\frac{1}{h_{23}} & 0 & -\frac{h_{13}}{h_{11}}\\
0 & \frac{1}{h_{23}} & \frac{h_{13}^2-h_{11}h_{33}}{2h_{11}h_{23}} \\
0 & 0 & 1 \end{matrix}\right]$} $\in\aut(\frakg_c)$ and
$C^t [h] C =$ {\footnotesize $\left[\begin{matrix}\mu & 0 & 0\\ 0 & 0 & 1\\ 0 & 1 & 0 \end{matrix}\right]$} where $\mu>0$.

Now suppose $h_{11}h_{22}-h_{12}^2\ne 0$.
First we may assume that $h_{13}=h_{23}=0$ because
$B =$
{\footnotesize $\left[\begin{matrix}
1 & 0 & \frac{h_{13}h_{22}-h_{12}h_{23}}{h_{12}^2-h_{11}h_{22}}\\
0 & 1 & \frac{h_{11}h_{23}-h_{12}h_{13}}{h_{12}^2-h_{11}h_{22}} \\
0 & 0 & 1 \end{matrix}\right]$} $\in\aut(\frakg_c)$ and
$B^t[h]B=[h']$ where $h'_{13}=h'_{23}=0.$

If $h_{11}\ne h_{12}$, then the quadratic equation
$$
(h_{11}-h_{12})z^2 + ((c-2)h_{11}+2h_{12}-h_{22})z - (c-1)(h_{11}-h_{12})=0
$$
has a real root, say $\beta$, because $-(c-1)(h_{11}-h_{12})^2<0$.
Then
$B =${\footnotesize
$\left[\begin{matrix}
\beta-1 & -c & 0\\
1 & \beta+1 &0 \\
0 & 0 & 1 \end{matrix}\right]$} $\in\aut(\frakg_c)$ and
$B^t[h]B=[h']$ where $h_{11}'=h_{12}'$ and $h_{13}'=h_{23}'=0$.
Thus we may also assume that $h_{11}=h_{12}$.
Since $\det[h_{ij}]=-h_{11}h_{33}(h_{11}-h_{22})<0$, we have
$h_{11}\ne 0$.
Let $B=\diag\{\frac{1}{\sqrt{|h_{11}|}},\frac{1}{\sqrt{|h_{11}|}},1\}\in\aut(\frakg_c)$.
Then
$B^t[h]B$ is of the form
$$
\left[\begin{matrix} 1 & 1 & 0\\ 1 & \tau & 0\\ 0 & 0 & \mu \end{matrix}\right], \quad
\left[\begin{matrix} -1 & -1 & 0\\ -1 & -\tau & 0\\ \hspace{8pt}0 & \hspace{8pt}0 & \mu \end{matrix}\right]
$$
where $\mu(\tau-1)<0$.

Now consider equivalences up to automorphism between the following three types of left invariant Lorentzian metrics :
\begin{align*}
\text{Type 1.}\, &\left[\begin{matrix}\mu & 0 & 0\\ 0 & 0 & 1\\ 0 & 1 & 0 \end{matrix}\right],  &
\text{Type 2.}\, &\left[\begin{matrix} 1 & 1 & 0\\ 1 & \tau & 0\\ 0 & 0 & \mu \end{matrix}\right],  &
\text{Type 3.}\, &\left[\begin{matrix} -1 & -1 & 0\\ -1 & -\tau & 0\\ \hspace{8pt}0 & \hspace{8pt}0 & \mu \end{matrix}\right]
\end{align*}

\begin{itemize}
  \item[(1)] Two different matrices of type 1 are not equivalent. And
every matrix of type 1 is not equivalent to that of type 2 or type 3.
  \item[(2)] Suppose $[h]=${\footnotesize $\left[\begin{matrix} 1 & 1 & 0\\ 1 & \tau & 0\\ 0 & 0 & \mu \end{matrix}\right]$},
  $[h']=${\footnotesize $\left[\begin{matrix} 1 & 1 & 0\\ 1 & \tau'& 0\\ 0 & 0 & \mu' \end{matrix}\right]$}
  and $[h]\sim[h']$. Then $B^t[h]B=[h']$ for some
$B =${\footnotesize
$\left[\begin{matrix}
\beta-\alpha & -c\alpha & x\\
\alpha & \beta+\alpha &y \\
0 & 0 & 1 \end{matrix}\right]$} $\in\aut(\frakg_c)$.
Thus we have
\begin{align}\label{eq-cge1}
&\beta x+ (\beta+\alpha(\tau-1))y=0,\tag{i}\\
&(\beta+\alpha-c\alpha)x + ((\beta+\alpha)\tau-c\alpha)y=0,\tag{ii}\\
&\beta^2-\alpha^2 + \tau \alpha^2 = 1,\tag{iii}\\
&\beta^2-\alpha^2 + \tau \alpha^2 + \alpha\beta(\tau-c)=1,\tag{iv}\\
&\tau \beta^2 + 2(\tau-c)\alpha\beta + (c^2-2c+\tau)\alpha^2=\tau',\notag\\
&\mu+x^2+2xy+\tau y^2=\mu' \notag
\end{align}
Considering the linear system of equations in the variables $x,y$ given by (i), (ii), the determinant of the coefficient matrix is $(\beta^2+(c-1)\alpha^2)(\tau-1)\ne0$. This implies $x=y=0$ and $\mu'=\mu$.
Using (iii) and (iv), we have $\alpha\beta(\tau-c)=0$.
\begin{itemize}
  \item[$\bullet$] If $\beta=0$, then $\tau>1$ and $\alpha=\pm\frac{1}{\sqrt{\tau-1}}$ and $\tau'=1+\frac{(c-1)^2}{\tau-1}$.
  \item[$\bullet$] If $\alpha=0$, then $\beta=\pm1$ and $\tau'=\tau$.
  \item[$\bullet$] If $\tau=c$, then $\beta^2+(c-1)\alpha^2=1$ and $\tau'=\tau=c$.
\end{itemize}
Hence we have
\begin{align*}
&\left[\begin{matrix} 1 & 1 & 0\\ 1 & \tau & 0\\ 0 & 0 & \mu \end{matrix}\right]
\sim \left[\begin{matrix} 1 & 1 & 0\\ 1 & \tau'& 0\\ 0 & 0 & \mu' \end{matrix}\right] \\
&\Longleftrightarrow \mu'=\mu,
\tau' = \left\{\begin{array}{ll}
\tau, &\text{if\,\,} \tau<1 \\
1+\frac{(c-1)^2}{\tau-1}, &\text{if\,\,} \tau>1
\end{array}
\right.
\end{align*}

  \item[(3)] Similarly,  we have
\begin{align*}
&\left[\begin{matrix} -1 & -1 & 0\\ -1 & -\tau & 0\\ \hspace{8pt}0 & \hspace{8pt}0 & \mu \end{matrix}\right]
\sim \left[\begin{matrix} -1 & -1 & 0\\ -1 & -\tau'& 0\\ 0 & 0 & \mu' \end{matrix}\right] \\
&\Longleftrightarrow \mu'=\mu,
\tau' = \left\{\begin{array}{ll}
\tau, &\text{if\,\,} \tau<1 \\
1+\frac{(c-1)^2}{\tau-1}, &\text{if\,\,} \tau>1
\end{array}
\right.
\end{align*}

  \item[(4)] Suppose $[h]=${\footnotesize$\left[\begin{matrix} -1 & -1 & 0\\ -1 & -\tau& 0\\ \hspace{8pt}0 & \hspace{8pt}0 & \mu \end{matrix}\right]$},
  $[h']=${\footnotesize$\left[\begin{matrix} 1 & 1 & 0\\ 1 & \tau' & 0\\ 0 & 0 & \mu' \end{matrix}\right]$}
  and $[h]\sim[h']$. Then there is $B =${\footnotesize
$\left[\begin{matrix}
\beta-\alpha & -c\alpha & x\\
\alpha & \beta+\alpha &y \\
0 & 0 & 1 \end{matrix}\right]$} $\in\aut(\frakg_c)$ such that $B^t[h]B=[h']$.
That is,
\begin{align*}
&\beta x+ (\beta+\alpha(\tau-1))y=0,\\
&(\beta+\alpha-c\alpha)x + ((\beta+\alpha)\tau-c\alpha)y=0,\\
&-\beta^2+\alpha^2 - \tau \alpha^2 = 1,\\
&-\beta^2+\alpha^2 - \tau \alpha^2 - \alpha\beta(\tau-c)=1,\\
&-\tau \beta^2 - 2(\tau-c)\alpha\beta - (c^2-2c+\tau)\alpha^2=\tau',\\
&\mu-x^2-2xy-\tau y^2 = \mu'
\end{align*}
This implies $x=y=0$, $\mu'=\mu$ and $\alpha\beta(\tau-c)=0$.
\begin{itemize}
  \item[$\bullet$] If $\beta=0$, then $\tau<1$ and $\alpha=\pm\frac{1}{\sqrt{1-\tau}}$ and $\tau'=1-\frac{(c-1)^2}{1-\tau}$.
  \item[$\bullet$] If $\alpha=0$, then $\beta^2+1=0$, which is a contradiction.
  \item[$\bullet$] If $\tau=c$, then $-\beta^2-(c-1)\alpha^2=1$, which is a contradiction.
\end{itemize}
Hence we have
\begin{align*}
&\left[\begin{matrix} -1 & -1 & 0\\ -1 & -\tau & 0\\ \hspace{8pt}0 & \hspace{8pt}0 & \mu \end{matrix}\right]
\sim \left[\begin{matrix} 1 & 1 & 0\\ 1 & \tau'& 0\\ 0 & 0 & \mu' \end{matrix}\right] \\
&\Longleftrightarrow \tau<1, \, \mu'=\mu, \, \tau' = 1-\frac{(c-1)^2}{1-\tau}
\end{align*}

  \item[(5)] Suppose $[h]=${\footnotesize$\left[\begin{matrix} -1 & -1 & 0\\ -1 & -\tau& 0\\ \hspace{8pt}0 & \hspace{8pt}0 & -\mu \end{matrix}\right]$} with $1<\tau\le c$ and $\mu>0$.
Let
$$
y_1=x, \,\, y_2=\frac{1}{\sqrt{\rho-1}}(x-y), \,\, y_3=\frac{1}{\sqrt{\mu}}z
$$
Then $\calB=\{y_1,y_2,y_3\}$ is a basis of the Lie algebra of $G_c$ such that
$[h]_{\calB} = \diag\{-1,-1,-1\}$, which is a contradiction.
\end{itemize}
Hence the proof is completed.
\end{proof}

For each left invariant Lorentzian metric $h$ on $G_c$ with $c>1$ whose associated matrix is given in Theorem~\ref{thm: c>1}, we obtain an orthonormal basis $\{y_i\}$ as follows:
\begin{itemize}
  \item [(1)] If $[h]=\left[\begin{matrix} \mu & 0 & 0\\ 0 & 0 & 1\\ 0 & 1 & 0 \end{matrix} \right]$ with $\mu> 0$,
\begin{align*}
  &y_1=\frac{1}{\sqrt{\mu}}x, \,
  y_2=\frac{1}{\sqrt{2}}(y+z), \,
  y_3=\frac{1}{\sqrt{2}}(y-z);\\
  &[y_1,y_2]=[y_3,y_1]=-\frac{1}{2\sqrt{\mu}}(y_2+y_3), \,
  [y_2,y_3]=-c\mu Y_1+\sqrt{2}(y_2+y_3).
\end{align*}

  \item [(2)] If $[h]=\left[\begin{matrix} 1 & 1 & 0\\ 1 & \tau & 0\\ 0 & 0 & \mu \end{matrix} \right]$ with $\mu> 0$, $\tau<1$,
\begin{align*}
&y_1=\frac{1}{\sqrt{\mu}}z, \,
 y_2=x, \,
 y_3=\frac{1}{\sqrt{1-\tau}}(x-y);\\
&[y_1,y_2]=\frac{1}{\sqrt{\mu}}y_2- \frac{\sqrt{1-\tau}}{\sqrt{\mu}}y_3, \,
[y_2,y_3]=0, \,
[y_3,y_1]=\frac{1-c}{\sqrt{(1-\tau)\mu}}y_2-\frac{1}{\sqrt{\mu}}y_3.
\end{align*}

  \item [(3)] If $[h]=${\footnotesize$\left[\begin{matrix} 1 & 1 & \hspace{8pt}0\\ 1 & \nu & \hspace{8pt}0\\ 0 & 0 & -\mu \end{matrix} \right]$} with $\mu> 0$, $1<\nu\le c$,
\begin{align*}
&y_1=x, \, y_2=\frac{1}{\sqrt{\nu-1}}(x-y), \,y_3=\frac{1}{\sqrt{\mu}}z;\\
&[y_1,y_2]=0, \,
[y_2,y_3]=\frac{1-c}{(\nu-1)\mu}y_1-\frac{1}{\sqrt{\mu}}y_2, \,
[y_3,y_1]=\frac{1}{\sqrt{\mu}}y_1-\frac{\sqrt{\nu-1}}{\sqrt{\mu}}y_2.
\end{align*}\

\end{itemize}

With respect to the orthonormal basis $\{y_i\}$, we have the following:

\begin{Thm}\label{thm: c>1-curvature}
The Ricci operator $\Ric$ of the metric $h$ on $G_c$ with $c>1$ is expressed as follows:
\begin{itemize}
  \item [(1)] If $[h]=\left[\begin{matrix} \mu & 0 & 0\\ 0 & 0 & 1\\ 0 & 1 & 0 \end{matrix} \right]$ with $\mu> 0$,
\begin{align*}
[\Ric]_{\{y_i\}}
&= \left[\begin{matrix} -\frac{c^2\mu}{2} & 0 & \hspace{8pt}0 \\ \hspace{8pt}0 & \frac{c(c\mu+1)}{2} & -\frac{c}{2} \\
\hspace{8pt}0 & \frac{c}{2} & \frac{c(c\mu-1)}{2} \end{matrix} \right] \quad(\text{O'Neill type $\{21\}$})
\end{align*}

  \item [(2)] If $[h]=\left[\begin{matrix} 1 & 1 & 0\\ 1 & \tau & 0\\ 0 & 0 & \mu \end{matrix} \right]$ with $\mu> 0$, $\tau<1$,
 \begin{align*}
[\Ric]_{\{y_i\}}
&= \left[\begin{matrix} \frac{\tau^2+(4-2c)\tau+(c^2-4)}{2(1-\tau)\mu} & 0 & 0 \\
0 &\hspace{-10pt} \frac{\tau^2+2\tau-(c^2-2c+4)}{2(1-\tau)\mu} & -\frac{c-\tau}{\mu\sqrt{1-\tau}} \\
0 & \frac{c-\tau}{\mu\sqrt{1-\tau}} &\hspace{-10pt}
  -\frac{\tau^2-6\tau-(c^2-2c-4)}{2(1-\tau)\mu} \end{matrix} \right]
\end{align*}
$\Ric$ is of O'Neill type $\{11,1\}$, $\{1z\bar{z}\}$, $\{21\}$ when $(c+\tau)^2-4c >0, <0, =0$, respectively.

  \item [(3)] If $[h]=\left[\begin{matrix} 1 & 1 & \hspace{8pt}0\\ 1 & \nu & \hspace{8pt}0\\ 0 & 0 & -\mu \end{matrix} \right]$ with $\mu> 0$, $1<\nu\le c$,
 \begin{align*}
[\Ric]_{\{y_i\}}
&= \left[\begin{matrix} -\frac{\nu^2+2\nu-(c^2-2c+4)}{2(\nu-1)\mu} & \frac{c-\nu}{\mu\sqrt{\nu-1}} & 0 \\
\frac{c-\nu}{\mu\sqrt{\nu-1}} &\hspace{-10pt} -\frac{\nu^2-6\nu-(c^2-2c-4)}{2(\nu-1)\mu} & 0 \\
0 & 0 &\hspace{-10pt} \frac{(\nu-c)^2+4(\nu-1)}{2(\nu-1)\mu} \end{matrix} \right]\\ &\quad(\text{O'Neill type $\{11,1\}$})
\end{align*}

\end{itemize}
\end{Thm}

\subsection{The case of $G_1$}\label{sec:c=1}

In this section, we consider the three-dimensional non-unimodular Lie group $G_1$,
and then we will compute the moduli space $\aut(G_1)\bs\frakM(G_1)$ and the Ricci operator.

Instead of the basis $\{x,y,z\}$ for $\frakg_1$ given in Section~\ref{sec: Lie algebras},
we choose a new basis $\{x_1,x_2,x_3\}$ for $\frakg_1$
by putting $x_1=-x+y, \,\, x_2=-x+2y,\,\,x_3=z$.
Then they satisfy
$$
[x_1,x_2]=0, \quad [x_3,x_1]=x_1, \quad [x_3,x_2]=x_1+x_2.
$$
With respect to this new basis, the Lie group $\aut(\frakg_1)$ is isomorphic to
$$
 \left\{\left[ \begin{matrix} \gamma&\delta& *\\0&\gamma& *\\ 0 & 0 & 1\end{matrix}\right]
\bigg|  \begin{array}{l} \gamma,\delta,*\in\bbr, \gamma\ne0\end{array} \right\}.
$$
In fact, given $\gamma,\delta$,
$$
\begin{array}{l}
\left[\begin{matrix}-2&-1&0\\\hspace{8pt}1&\hspace{8pt}1&0\\ \hspace{8pt}0&\hspace{8pt}0&1\end{matrix}\right]^{-1}
\left[ \begin{matrix} \gamma&\delta& *\\0&\gamma& *\\ 0 & 0 & 1\end{matrix}\right]
\left[\begin{matrix}-2&-1&0\\\hspace{8pt}1&\hspace{8pt}1&0\\ \hspace{8pt}0&\hspace{8pt}0&1\end{matrix}\right]
=\left[\begin{matrix} \gamma-\delta & -\delta & * \\\delta & \gamma+\delta & *\\ 0 & 0 &
1\end{matrix}\right]
\end{array}
$$
where $[\id]_{\{x,y,z\},\{x_1,x_2,x_3\}} = \left[\begin{matrix}-2&-1&0\\\hspace{8pt}1&\hspace{8pt}1&0\\ \hspace{8pt}0&\hspace{8pt}0&1\end{matrix}\right]$.

\begin{Thm}\label{thm: c=1}
Any left invariant Lorentzian metric on $G_1$ is equivalent up
to automorphism to a metric whose associated matrix is of the form
$$
\left[\begin{matrix}-2&-1&0\\\hspace{8pt}1&\hspace{8pt}1&0\\ \hspace{8pt}0&\hspace{8pt}0&1\end{matrix}\right]^{-1}
M \left[\begin{matrix}-2&-1&0\\\hspace{8pt}1&\hspace{8pt}1&0\\ \hspace{8pt}0&\hspace{8pt}0&1\end{matrix}\right]
$$
where $M$ is of the form
\begin{align*}
&\left[\begin{matrix}0 & 0 & 1\\ 0 & \mu & 0\\ 1 & 0 & 0 \end{matrix}\right],
&& \left[\begin{matrix}\mu & 0 & 0\\ 0 & 0 & 1\\ 0 & 1 & 0 \end{matrix}\right],
&&\left[\begin{matrix} 1 & \hspace{8pt}0 & 0\\ 0 & -\nu & 0\\ 0 & \hspace{8pt}0 & \mu \end{matrix}\right],
&&\left[\begin{matrix} 1 & 0 & \hspace{8pt}0\\ 0 & \nu & \hspace{8pt}0\\ 0 & 0 & -\mu \end{matrix}\right],\\
&\left[\begin{matrix} -1 & 0 & 0\\ \hspace{8pt}0 & \nu & 0\\ \hspace{8pt}0 & 0 & \mu \end{matrix}\right],
&&\left[\begin{matrix} 0 & 1 & 0\\ 1 & 0 & 0\\ 0 & 0 & \mu \end{matrix}\right],
&&\left[\begin{matrix} \hspace{8pt}0 & -1 & 0\\ -1 & \hspace{8pt}0 & 0\\ \hspace{8pt}0 & \hspace{8pt}0 & \mu \end{matrix}\right]
\end{align*}
where $\mu>0$ and $\nu>0$.
\end{Thm}

\begin{proof}
Let $\{x_1,x_2,x_3\}$ be the basis for $\frakg_c$ given by
$$x_1=-x+y, \,\, x_2=-x+2y,\,\,x_3=z$$
Then 
$$[x_1,x_2]=0, \quad [x_3,x_1]=x_1, \quad [x_3,x_2]=x_1+x_2$$ and
the Lie group $\aut(\frakg_1)$ is isomorphic to {\footnotesize $
\left\{\left[\begin{matrix} \gamma&\delta& *\\0&\gamma& *\\ 0 & 0 &
1\end{matrix}\right] \bigg|\,\, \gamma,\delta,*\in\bbr,
\gamma\ne 0 \right\}$}.

Let $h$ be a left invariant Lorentzian metric on $G_1$.
Then with respect to the basis $\{x_1,x_2,x_3\}$,
$[h]=[h_{ij}]$ is a symmetric and non-degenerate real $3\times 3$ matrix.

First assume that  $h_{11}h_{22}-h_{12}^2= 0$.
If $h_{12}\ne 0$, then $h_{11}\ne0, h_{22}\ne 0$.
Let  $B =${\footnotesize
$\left[\begin{matrix}
h_{11} & -h_{12} & 0\\
0 & \hspace{8pt}h_{11} & 0 \\
0 & \hspace{8pt}0 & 1 \end{matrix}\right]$} $\in\aut(\frakg_c)$. Then
 $B^t [h] B =[h_{ij}']$ where $h_{12}'=0$.
Thus we may assume that $h_{12}=0$.
Then either $h_{11}=0, h_{22}\ne 0$ or $h_{11}\ne 0, h_{22}= 0$.
When $h_{11}=0$ and $h_{22}\ne 0$, $\det([h])=-h_{22}h_{13}^2<0$ and so
we have $h_{22}>0$ and $h_{13}\ne 0$.
Let  $B =${\footnotesize
$\left[\begin{matrix}
\frac{1}{h_{13}} & 0 & 0\\
0 & \frac{1}{h_{13}}& -\frac{h_{23}}{h_{22}} \\
0 & 0 & 1 \end{matrix}\right]$} $\in\aut(\frakg_c)$. Then
 $B^t [h] B =$ {\footnotesize$\left[\begin{matrix}0 & 0 & 1\\ 0 & \mu & 0\\ 1 & 0 & \nu  \end{matrix}\right]$}.
Let $C =${\footnotesize
$\left[\begin{matrix}
1 & 0 & -\frac{\nu}{2}\\
0 & 1 & 0 \\
0 & 0 & 1 \end{matrix}\right]$} $\in\aut(\frakg_c)$. Then
 $C^tB^t [h] BC =$ {\footnotesize$\left[\begin{matrix}0 & 0 & 1\\ 0 & \mu & 0\\ 1 & 0 & 0  \end{matrix}\right]$}  where $\mu>0$.

Similarly, when $h_{11}\ne 0$ and $h_{22}=0$, then $h_{23}\ne 0$,
$B =${\footnotesize
$\left[\begin{matrix}
\frac{1}{h_{23}} & 0 & -\frac{h_{13}}{h_{11}}\\
0 & \frac{1}{h_{23}} & \hspace{8pt}0 \\
0 & 0 & \hspace{8pt}1 \end{matrix}\right]$} $\in\aut(\frakg_c)$ and
$B^t [h] B =$ {\footnotesize $\left[\begin{matrix}\mu & 0 & 0\\ 0 & 0 & 1\\ 0 & 1 & \nu \end{matrix}\right]$}.
And let $C =${\footnotesize
$\left[\begin{matrix}
1 & 0 & \hspace{8pt}0\\
0 & 1 & -\frac{\nu}{2} \\
0 & 0 & \hspace{8pt}1 \end{matrix}\right]$} $\in\aut(\frakg_c)$. Then
 $C^tB^t [h] BC =$ {\footnotesize$\left[\begin{matrix}\mu & 0 & 0\\ 0 & 0 & 1\\ 0 & 1 & 0  \end{matrix}\right]$}  where $\mu>0$.

Next suppose that $h_{11}h_{22}-h_{12}^2\ne 0$.
Then we may assume that $h_{13}=h_{23}=0$ because
$B =$
{\footnotesize $\left[\begin{matrix}
1 & 0 & \frac{h_{13}h_{22}-h_{12}h_{23}}{h_{12}^2-h_{11}h_{22}}\\
0 & 1 & \frac{h_{11}h_{23}-h_{12}h_{13}}{h_{12}^2-h_{11}h_{22}} \\
0 & 0 & 1 \end{matrix}\right]$} $\in\aut(\frakg_c)$ and
$B^t[h]B=[h']$ where $h'_{13}=h'_{23}=0.$

If $h_{11}\ne 0$, then
$B=$ {\footnotesize $\left[\begin{matrix} h_{11} & -h_{12} & 0 \\ 0 & \hspace{8pt}h_{11} & 0 \\0 & \hspace{8pt}0 & 1 \end{matrix} \right]$} $\in\aut(\frakg_c)$ and $B^t [h] B = \diag\{d_1,d_2,d_3\}$ where $d_1d_2d_3<0$.
Since only one of $d_1, d_2, d_3$ is negative,
either $d_1 d_2<0, \lambda>0$ or $d_1>0, d_2>0, \lambda<0$.
Let $C=\diag\{\frac{1}{\sqrt{|d_1|}}, \frac{1}{\sqrt{|d_1|}}, 1 \}\in\aut(\frakg_c)$.
Then $C^tB^t[h]BC$ is of the form
$$
\left[\begin{matrix} 1 & \hspace{8pt}0 & 0\\ 0 & -\nu & 0\\ 0 & \hspace{8pt}0 & \mu \end{matrix}\right], \quad
\left[\begin{matrix} 1 & 0 & \hspace{8pt}0\\ 0 & \nu & \hspace{8pt}0\\ 0 & 0 & -\mu \end{matrix}\right], \quad
\left[\begin{matrix} -1 & 0 & 0\\ \hspace{8pt}0 & \nu & 0\\ \hspace{8pt}0 & 0 & \mu \end{matrix}\right]
$$
where $\mu >0$ and $\nu>0$.

If $h_{11}=0$, then $\det([h])=-h_{33}(h_{12})^2$. Thus $h_{33}>0$ and $h_{12}\ne0$.
Let $B=\diag\{\frac{1}{\sqrt{|h_{12}|}}, \frac{1}{\sqrt{|h_{12}|}}, 1\}\in\aut(\frakg_c)$.
Then
$B^t[h]B=$ {\footnotesize $\left[\begin{matrix} \hspace{8pt}0 & \pm1 & 0\\ \pm1 & \hspace{8pt}\nu & 0\\ \hspace{8pt}0 & \hspace{8pt}0 & \mu \end{matrix}\right]$}.
Let $C =${\footnotesize
$\left[\begin{matrix}
1 & \mp\frac{\nu}{2} & 0\\
0 & \hspace{8pt}1 & 0 \\
0 & \hspace{8pt}0 & 1 \end{matrix}\right]$} $\in\aut(\frakg_c)$. Then
$C^tB^t [h] BC$ is of the form
$$
\left[\begin{matrix}0 & 1 & 0\\ 1 & 0& 0\\ 0 & 0 & \mu  \end{matrix}\right], \quad
\left[\begin{matrix}\hspace{8pt}0 & -1 & 0\\ -1 & \hspace{8pt}0& 0\\ \hspace{8pt}0 & \hspace{8pt}0 & \mu  \end{matrix}\right]
$$
where $\mu>0$.
Finally it is easy to see that any four such distinct matrices are not equivalent.
\end{proof}

Let $Q=${$\left[\begin{matrix}-2&-1&0\\\hspace{8pt}1&\hspace{8pt}1&0\\ \hspace{8pt}0&\hspace{8pt}0&1\end{matrix}\right]$}.
For each left invariant Lorentzian metric $h$ on $G_1$ whose associated matrix is given in Theorem~\ref{thm: c=1}, we obtain an orthonormal basis $\{y_i\}$ as follows:
\begin{itemize}
  \item [(1)] If $[h]=Q^{-1}${$\left[\begin{matrix} 0 & 0 & 1\\ 0 & \mu & 0\\ 1 & 0 & 0 \end{matrix} \right]$}$Q$ with $\mu>0$,
\begin{align*}
  &y_1=\frac{1}{\sqrt{\mu}}x_2, \,
  y_2=\frac{1}{\sqrt{2}}(x_1+x_3), \,
  y_3=\frac{1}{\sqrt{2}}(x_1-x_3);\\
 &[y_1,y_2]=[y_3,y_1]=-\frac{1}{\sqrt{2}}y_1-\frac{1}{2\sqrt{\mu}}(y_2+y_3), \,
  [y_2,y_3]=\frac{1}{\sqrt{2}}(y_2+y_3).
\end{align*}

  \item [(2)] If $[h]=Q^{-1}${$
  \left[\begin{matrix} \mu & 0 & 0\\ 0 & 0 & 1\\ 0 & 1 & 0 \end{matrix} \right]$}$Q$  with $\mu>0$,
\begin{align*}
  &y_1=\frac{1}{\sqrt{\mu}}X_1, \,
  y_2=\frac{1}{\sqrt{2}}(X_2+X_3), \,
  y_3=\frac{1}{\sqrt{2}}(X_2-X_3);\\
  &[y_1,y_2]=[y_3,y_1]=-\frac{1}{\sqrt{2}} Y_1, \,
  [y_2,y_3]= \sqrt{\mu}Y_1+\frac{1}{\sqrt{2}}(Y_2+Y_3).
\end{align*}

  \item [(3)] If $[h]=Q^{-1}${$
  \left[\begin{matrix} 1 & \hspace{8pt}0 & 0\\ 0 & -\nu & 0\\ 0 & \hspace{8pt}0 & \mu \end{matrix} \right]$}$Q$
  with $\mu>0$, $\nu>0$,
\begin{align*}
&y_1=x_1, \, y_2=\frac{1}{\sqrt{\mu}}x_3, \, y_3=\frac{1}{\sqrt{\nu}}x_2;\\
&[y_1,y_2]=-\frac{1}{\sqrt{\mu}}y_1, \,
[y_2,y_3]=\frac{1}{\sqrt{\mu\nu}}y_1 + \frac{1}{\sqrt{\mu}}y_3, \,
[y_3,y_1]=0.
\end{align*}

  \item [(4)] If $[h]=Q^{-1}${$
  \left[\begin{matrix} 1 & 0 & \hspace{8pt}0\\ 0 & \nu & \hspace{8pt}0\\ 0 & 0 & -\mu \end{matrix} \right]$}$Q$  with $\mu>0$, $\nu>0$,
\begin{align*}
&y_1=x_1, \,  y_2=\frac{1}{\sqrt{\nu}}x_2, \,  y_3=\frac{1}{\sqrt{\mu}}x_3;\\
&[y_1,y_2]=0, \,
[y_2,y_3]=-\frac{1}{\sqrt{\mu\nu}}y_1 - \frac{1}{\sqrt{\mu}}y_2, \,
[y_3,y_1]=\frac{1}{\sqrt{\mu}}y_1.
\end{align*}

  \item [(5)] If $[h]=Q^{-1}${$
  \left[\begin{matrix} -1 & 0 & 0\\ \hspace{8pt}0 & \nu & 0\\ \hspace{8pt}0 & 0 & \mu \end{matrix} \right]$}$Q$  with $\mu>0$, $\nu>0$,
\begin{align*}
&y_1=\frac{1}{\sqrt{\nu}}x_2, \,
  y_2=\frac{1}{\sqrt{\mu}}x_3, \,
  y_3=x_1;\\
  &[y_1,y_2]=- \frac{1}{\sqrt{\mu}}y_1 -\frac{1}{\sqrt{\mu\nu}}y_3, \,
[y_2,y_3]= \frac{1}{\sqrt{\mu}}y_3, \,
[y_3,y_1]= 0.
\end{align*}

 \item [(6)] If $[h]=Q^{-1}${$
 \left[\begin{matrix} 0 & 1 & 0\\ 1 & 0 & 0\\ 0 & 0 & \mu \end{matrix} \right]$}$Q$ with $\mu>0$,
\begin{align*}
  &y_1=\frac{1}{\sqrt{\mu}}x_3, \,
  y_2=\frac{1}{\sqrt{2}}(x_1+x_2), \,
  y_3=\frac{1}{\sqrt{2}}(x_1-x_2);\\
  &[y_1,y_2]=\frac{1}{2\sqrt{\mu}}(3y_2+y_3), \,
  [y_2,y_3]=0, \,
  [y_3,y_1]=\frac{1}{2\sqrt{\mu}}(y_2-y_3).
\end{align*}

 \item [(7)] If $[h]=Q^{-1}${$
 \left[\begin{matrix} \hspace{8pt}0 & -1 & 0\\ -1 & \hspace{8pt}0 & 0\\ \hspace{8pt}0 & \hspace{8pt}0 & \mu \end{matrix} \right]$}$Q$ with $\mu>0$,
\begin{align*}
  &y_1=\frac{1}{\sqrt{\mu}}x_3, \,
  y_2=\frac{1}{\sqrt{2}}(x_1-x_2), \,
  y_3=\frac{1}{\sqrt{2}}(x_1+x_2);\\
  &[y_1,y_2]=\frac{1}{2\sqrt{\mu}}(y_2-y_3), \,
  [y_2,y_3]=0, \,
  [y_3,y_1]=-\frac{1}{2\sqrt{\mu}}(y_2+3y_3).
\end{align*}

\end{itemize}

With respect to the orthonormal basis $\{y_i\}$, we have the following:

\begin{Thm}\label{thm: c=1-curvature}
The Ricci operator $\Ric$ of the metric $h$ on $G_1$ is expressed as follows:
\begin{itemize}
  \item [(1)] If $[h]=\left[\begin{matrix}-2&-1&0\\\hspace{8pt}1&\hspace{8pt}1&0\\ \hspace{8pt}0&\hspace{8pt}0&1\end{matrix}\right]^{-1}
  \left[\begin{matrix} 0 & 0 & 1\\ 0 & \mu & 0\\ 1 & 0 & 0 \end{matrix} \right]
  \left[\begin{matrix}-2&-1&0\\\hspace{8pt}1&\hspace{8pt}1&0\\ \hspace{8pt}0&\hspace{8pt}0&1\end{matrix}\right]$ with $\mu>0$
\begin{align*}
[\Ric]_{\{y_i\}}
&=\diag\{0,0,0\} \,\text{(flat)} \quad(\text{O'Neill type $\{11,1\}$})
\end{align*}

  \item [(2)] If $[h]=\left[\begin{matrix}-2&-1&0\\\hspace{8pt}1&\hspace{8pt}1&0\\ \hspace{8pt}0&\hspace{8pt}0&1\end{matrix}\right]^{-1}
  \left[\begin{matrix} \mu & 0 & 0\\ 0 & 0 & 1\\ 0 & 1 & 0 \end{matrix} \right]
  \left[\begin{matrix}-2&-1&0\\\hspace{8pt}1&\hspace{8pt}1&0\\ \hspace{8pt}0&\hspace{8pt}0&1\end{matrix}\right]$ with $\mu>0$,
 \begin{align*}
[\Ric]_{\{y_i\}}
&= \left[\begin{matrix} -\frac{\mu}{2} & \sqrt{\frac{\mu}{2}} & -\sqrt{\frac{\mu}{2}} \\
\sqrt{\frac{\mu}{2}} & \frac{\mu}{2} & 0 \\
\sqrt{\frac{\mu}{2}} & 0 & \frac{\mu}{2}  \end{matrix} \right] \quad(\text{O'Neill type $\{21\}$})
\end{align*}

  \item [(3)] If $[h]=\left[\begin{matrix}-2&-1&0\\\hspace{8pt}1&\hspace{8pt}1&0\\ \hspace{8pt}0&\hspace{8pt}0&1\end{matrix}\right]^{-1}
  \left[\begin{matrix} 1 & \hspace{8pt}0 & 0\\ 0 & -\nu & 0\\ 0 & \hspace{8pt}0 & \mu \end{matrix} \right]
  \left[\begin{matrix}-2&-1&0\\\hspace{8pt}1&\hspace{8pt}1&0\\ \hspace{8pt}0&\hspace{8pt}0&1\end{matrix}\right]$ with $\mu>0$, $\nu>0$,
 \begin{align*}
[\Ric]_{\{y_i\}}
&= \left[\begin{matrix} -\frac{4\nu+1}{2\mu\nu} & 0 & -\frac{1}{\mu\sqrt{\nu}} \\
0 & \frac{1-4\nu}{2\mu\nu} & 0 \\
\hspace{8pt}\frac{1}{\mu\sqrt{\nu}} & 0 & \hspace{8pt}\frac{1-4\nu}{2\mu\nu} \end{matrix} \right]
\end{align*}
$\Ric$ is of O'Neill type $\{11,1\}$, $\{1z\bar{z}\}$, $\{21\}$  when $1-4\nu >0, <0, =0$, respectively.

  \item [(4)] If $[h]=\left[\begin{matrix}-2&-1&0\\\hspace{8pt}1&\hspace{8pt}1&0\\ \hspace{8pt}0&\hspace{8pt}0&1\end{matrix}\right]^{-1}
  \left[\begin{matrix} 1 & 0 & \hspace{8pt}0\\ 0 & \nu & \hspace{8pt}0\\ 0 & 0 & -\mu \end{matrix} \right]
  \left[\begin{matrix}-2&-1&0\\\hspace{8pt}1&\hspace{8pt}1&0\\ \hspace{8pt}0&\hspace{8pt}0&1\end{matrix}\right]$ with $\mu>0$, $\nu>0$,
 \begin{align*}
[\Ric]_{\{y_i\}}
&= \left[\begin{matrix} \frac{4\nu-1}{2\mu\nu} & \frac{1}{\mu\sqrt{\nu}}& 0 \\
\frac{1}{\mu\sqrt{\nu}} & \frac{4\mu+1}{2\mu\nu} & 0 \\
0 & 0 & \frac{4\mu+1}{2\mu\nu} \end{matrix} \right] \quad(\text{O'Neill type $\{11,1\}$})
\end{align*}

 \item [(5)] If $[h]=\left[\begin{matrix}-2&-1&0\\\hspace{8pt}1&\hspace{8pt}1&0\\ \hspace{8pt}0&\hspace{8pt}0&1\end{matrix}\right]^{-1}
 \left[\begin{matrix} -1 & 0 & 0\\ \hspace{8pt}0 & \nu & 0\\ \hspace{8pt}0 & 0 & \mu \end{matrix} \right]
 \left[\begin{matrix}-2&-1&0\\\hspace{8pt}1&\hspace{8pt}1&0\\ \hspace{8pt}0&\hspace{8pt}0&1\end{matrix}\right]$ with $\mu>0$, $\nu>0$,
 \begin{align*}
[\Ric]_{\{y_i\}}
&= \left[\begin{matrix} \hspace{8pt}\frac{1-4\nu}{2\mu\nu} & 0 & \frac{1}{\mu\sqrt{\nu}} \\
\hspace{8pt}0   & \frac{1-4\mu}{2\mu\nu} & 0 \\
-\frac{1}{\mu\sqrt{\nu}} & 0 & -\frac{4\mu+1}{2\mu\nu} \end{matrix} \right]
\end{align*}
$\Ric$ is of O'Neill type $\{11,1\}$, $\{1z\bar{z}\}$, $\{21\}$  when $1-4\nu >0, <0, =0$, respectively.

 \item [(6)] If $[h]=\left[\begin{matrix}-2&-1&0\\\hspace{8pt}1&\hspace{8pt}1&0\\ \hspace{8pt}0&\hspace{8pt}0&1\end{matrix}\right]^{-1}
 \left[\begin{matrix} 0 & 1 & 0\\ 1 & 0 & 0\\ 0 & 0 & \mu \end{matrix} \right]
 \left[\begin{matrix}-2&-1&0\\\hspace{8pt}1&\hspace{8pt}1&0\\ \hspace{8pt}0&\hspace{8pt}0&1\end{matrix}\right]$ with $\mu>0$,
 \begin{align*}
[\Ric]_{\{y_i\}}
&= \left[\begin{matrix} -\frac{2}{\mu} &\hspace{8pt} 0 & \hspace{8pt}0 \\
\hspace{8pt}0 & -\frac{3}{\mu} & \hspace{8pt}\frac{1}{\mu} \\
\hspace{8pt}0 & -\frac{1}{\mu} & -\frac{1}{\mu}  \end{matrix} \right] \quad(\text{O'Neill type $\{21\}$})
\end{align*}

 \item [(7)] If $[h]=\left[\begin{matrix}-2&-1&0\\\hspace{8pt}1&\hspace{8pt}1&0\\ \hspace{8pt}0&\hspace{8pt}0&1\end{matrix}\right]^{-1}
 \left[\begin{matrix} \hspace{8pt}0 & -1 & 0\\ -1 & \hspace{8pt}0 & 0\\ \hspace{8pt}0 & \hspace{8pt}0 & \mu \end{matrix} \right]
 \left[\begin{matrix}-2&-1&0\\\hspace{8pt}1&\hspace{8pt}1&0\\ \hspace{8pt}0&\hspace{8pt}0&1\end{matrix}\right]$ with $\mu>0$,
\begin{align*}
[\Ric]_{\{y_i\}}
&= \left[\begin{matrix} -\frac{2}{\mu} & \hspace{8pt}0 & \hspace{8pt}0 \\
\hspace{8pt}0 & -\frac{1}{\mu} & -\frac{1}{\mu} \\
\hspace{8pt}0 & \hspace{8pt}\frac{1}{\mu} & -\frac{3}{\mu}  \end{matrix} \right] \quad(\text{O'Neill type $\{21\}$})
\end{align*}

\end{itemize}
\end{Thm}

\begin{proof} Applying Lemma~\ref{lemma:OT1} and ~\ref{lemma:OT2},
except for (2), the O'Neill types  in all cases can be easily obtained.
(2) Let $[\Ric]_{\{y_i\}}=${\footnotesize$\left[\begin{matrix} -\frac{\mu}{2} & \sqrt{\frac{\mu}{2}} & -\sqrt{\frac{\mu}{2}} \\
\hspace{8pt}\sqrt{\frac{\mu}{2}} & \frac{\mu}{2} & \hspace{8pt}0 \\
\hspace{8pt}\sqrt{\frac{\mu}{2}} & 0 & \hspace{8pt}\frac{\mu}{2}  \end{matrix} \right]$}.
Then $[\Ric]_{\{y_i\}}$  has two eigenvalues $-\frac{2}{\mu},\frac{2}{\mu}$ ($\frac{2}{\mu}$ has multiplicity two), each associated to a one-dimensional eigenspace;
  Hence $\Ric$ is of O'Neill type $\{21\}$.
In this case, if we choose an orthonormal basis $\frakB'$ given by
$$
[\id]_{\frakB',\frakB} =
\left[\begin{matrix} \hspace{8pt}1&\hspace{8pt}\frac{1}{\sqrt{\mu}}&\hspace{8pt}\frac{1}{\sqrt{\mu}}\\
-\frac{1}{\sqrt{2\mu}}&\hspace{8pt}\frac{3\mu-1}{2\sqrt{2}\mu}&\hspace{8pt}\frac{\mu-1}{2\sqrt{2}\mu}\\
-\frac{1}{\sqrt{2\mu}}&-\frac{\mu-1}{2\sqrt{2}\mu}&-\frac{3\mu+1}{2\sqrt{2}\mu}\end{matrix}\right]\in O(2,1),
$$
then
$$
[\Ric]_{\frakB'} = \left[\begin{matrix}-\frac{\mu}{2}&\hspace{8pt}0&0\\\hspace{8pt}0&\frac{\mu}{2}+1&1\\\hspace{8pt}0&-1&\frac{\mu}{2}-1\end{matrix}\right]
$$
\end{proof}

\subsection{The case of $G_c$, $c<1$}\label{sec:c<1}

In this section, we consider the three-dimensional non-unimodular Lie group $G_c$ with $c<1$,
and then we will compute the moduli space $\aut(G_c)\bs\frakM(G_c)$ and the Ricci operator.

Instead of the basis $\{x,y,z\}$ for $\frakg_c$ given in Section~\ref{sec: Lie algebras},
we choose a new basis $\{x_1,x_2,x_3\}$ for $\frakg_c$
by putting $x_1=-(1-w)x+y, \,\, x_2=-(1+w)x+y,\,\,x_3=z$ where $w=\sqrt{1-c}$.
Then they satisfy
$$
[x_1,x_2]=0, \quad [x_3,x_1]=(1+w)x_1, \quad [x_3,x_2]=(1-w)x_2.
$$
With respect to this new basis, the Lie group $\aut(\frakg_c)$ is
isomorphic to
$$
 \left\{\left[ \begin{matrix} \gamma&0& *\\0&\delta& *\\ 0 & 0 & 1\end{matrix}\right]
\bigg|  \begin{array}{l} \gamma,\delta,*\in\bbr, \gamma\delta\ne
0\end{array} \right\}.
$$
In fact, given $\gamma,\delta$,
$$
\begin{array}{l}
\left[\begin{matrix}\hspace{8pt}\frac{1}{2w}&\frac{1+w}{2w}&0\\-\frac{1}{2w}&\frac{w-1}{2w}&0\\ \hspace{8pt}0&0&1\end{matrix}\right]^{-1}
\left[ \begin{matrix} \gamma&0&*\\0&\delta& *\\ 0 & 0 & 1\end{matrix}\right]
\left[\begin{matrix}\hspace{8pt}\frac{1}{2w}&\frac{1+w}{2w}&0\\-\frac{1}{2w}&\frac{w-1}{2w}&0\\ \hspace{8pt}0&0&1\end{matrix}\right]
=\left[\begin{matrix} \beta-\alpha & -c\alpha & * \\\alpha & \beta+\alpha & *\\ 0 & 0 &
1\end{matrix}\right]
\end{array}
$$
where $\alpha=\frac{\delta-\gamma}{2z}, \beta=\frac{\delta+\gamma}{2}$ and
$[\id]_{\{x,y,z\},\{x_1,x_2,x_3\}} = \left[\begin{matrix}\hspace{8pt}\frac{1}{2w}&\frac{1+w}{2w}&0\\-\frac{1}{2w}&\frac{w-1}{2w}&0\\ \hspace{8pt}0&0&1\end{matrix}\right]$.

\begin{Thm}\label{thm: c<1}
Any left invariant Lorentzian metric on $G_c$ with $c<1$ is equivalent up
to automorphism to a metric whose associated matrix is of the form
$$
\left[\begin{matrix}\hspace{8pt}\frac{1}{2w}&\frac{1+w}{2w}&0\\-\frac{1}{2w}&\frac{w-1}{2w}&0\\ \hspace{8pt}0&0&1\end{matrix}\right]^{-1}
M
\left[\begin{matrix}\hspace{8pt}\frac{1}{2w}&\frac{1+w}{2w}&0\\-\frac{1}{2w}&\frac{w-1}{2w}&0\\ \hspace{8pt}0&0&1\end{matrix}\right]
$$
where $M$ is of the form
\begin{align*}
&\left[\begin{matrix}0 & 0 & 1\\ 0 & 1 & 0\\ 1 & 0 & 0 \end{matrix}\right],
&& \left[\begin{matrix}1 & 0 & 0\\ 0 & 0 & 1\\ 0 & 1 & 0 \end{matrix}\right],
&&\left[\begin{matrix} 1 & 1 & 0\\ 1 & 1 & \mu\\ 0 & \mu & 0 \end{matrix}\right],
&&\left[\begin{matrix} 1 & 0 & \hspace{8pt}0\\ 0 & 1 & \hspace{8pt}0\\ 0 & 0 & -\mu \end{matrix}\right], \\
&\left[\begin{matrix} 1 & \hspace{8pt}0 & 0\\ 0 & -1 & 0\\ 0 & \hspace{8pt}0 & \mu \end{matrix}\right],
&&\left[\begin{matrix} -1 & 0 & 0\\ \hspace{8pt}0 & 1 & 0\\ \hspace{8pt}0 & 0 & \mu \end{matrix}\right],
&&\left[\begin{matrix}0 & 1 & 0\\ 1 & 0 & 0\\ 0 & 0 & \mu \end{matrix}\right],
&&\left[\begin{matrix}0 & 1 & 0\\ 1 & 1 & 0\\ 0 & 0 & \mu \end{matrix}\right],\\
&\left[\begin{matrix}0 & \hspace{8pt}1 & 0\\ 1 & -1 & 0\\ 0 & \hspace{8pt}0 & \mu \end{matrix}\right],
&&\left[\begin{matrix} 1 & 1 & 0\\ 1 & \tau & 0\\ 0 & 0 & \nu \end{matrix}\right],
&&\left[\begin{matrix} -1 & \hspace{8pt}1 & 0\\ \hspace{8pt}1 & -\eta & 0\\ \hspace{8pt}0 & \hspace{8pt}0 & \mu \end{matrix}\right]
\end{align*}
where $w=\sqrt{1-c},\,  \mu>0, \, \nu(\tau-1)<0$ and $\eta<1$.
\end{Thm}

\begin{proof}
Let $\{x_1,x_2,x_3\}$ be the basis for $\frakg_c$ given by
$$x_1=(-1+w)x+y, \,\, x_2=-(1+w)x+y,\,\,x_3=z$$
where $w=\sqrt{1-c}$.
Then
$$
[x_1,x_2]=0, \,\,[x_3,x_1]=(1+w)x_1, \,\,[x_3,x_2]=(1-w)x_2
$$
and
the Lie group $\aut(\frakg_c)$ is isomorphic to {\footnotesize $
\left\{\left[\begin{matrix} \gamma&0& *\\0&\delta& *\\ 0 & 0 &
1\end{matrix}\right] \bigg|\,\, \gamma,\delta,*\in\bbr,
\gamma\delta\ne 0 \right\}$}.

Let $h$ be a left invariant Lorentzian metric on $G_c$  with $c<1$.
Then with respect to the basis $\{x_1,x_2,x_3\}$,
$[h]=[h_{ij}]$ is a symmetric and non-degenerate real $3\times 3$ matrix.

First assume that  $h_{11}h_{22}-h_{12}^2= 0$ and $h_{12}=0$. Then either $h_{11}=0, h_{22}\ne 0$ or $h_{11}\ne 0, h_{22}= 0$.
Assume $h_{11}=0$ and $h_{22}\ne 0$. Since $\det([h])=-h_{22}h_{13}^2<0$,
we have $h_{22}>0$ and $h_{13}\ne 0$.
Let  $B =${\footnotesize
$\left[\begin{matrix}
\frac{1}{h_{13}} & 0 & 0\\
0 & \frac{1}{h_{13}}& -\frac{h_{23}}{h_{22}} \\
0 & 0 & 1 \end{matrix}\right]$} $\in\aut(\frakg_c)$. Then
 $B^t [h] B =$ {\footnotesize$\left[\begin{matrix}0 & 0 & 1\\ 0 & \mu & 0\\ 1 & 0 & \nu  \end{matrix}\right]$}.
 Since $\det(B^t [h] B)=-\mu<0$, we have $\mu>0$. Let $C =${\footnotesize
$\left[\begin{matrix}
1 & 0 & -\frac{\nu}{2}\\
0 & \frac{1}{\sqrt{\mu}} & \hspace{8pt}0 \\
0 & 0 & \hspace{8pt}1 \end{matrix}\right]$} $\in\aut(\frakg_c)$. Then
 $C^tB^t [h] BC =$ {\footnotesize$\left[\begin{matrix}0 & 0 & 1\\ 0 & 1 & 0\\ 1 & 0 & 0  \end{matrix}\right]$}.  Similarly, when $h_{11}\ne 0$ and $h_{22}=0$, then $h_{23}\ne 0$,
$B =${\footnotesize
$\left[\begin{matrix}
\frac{1}{h_{23}} & 0 & -\frac{h_{13}}{h_{11}}\\
0 & \frac{1}{h_{23}} & 0 \\
0 & 0 & 1 \end{matrix}\right]$} $\in\aut(\frakg_c)$ and
$B^t [h] B =$ {\footnotesize $\left[\begin{matrix}\mu & 0 & 0\\ 0 & 0 & 1\\ 0 & 1 & \nu \end{matrix}\right]$}.
 Since $\det(B^t [h] B)=-\mu<0$, we have $\mu>0$.
Let $C =${\footnotesize
$\left[\begin{matrix}
\frac{1}{\sqrt{\mu}} & 0 & \hspace{8pt}0\\
0 & 1 & -\frac{\nu}{2} \\
0 & 0 & \hspace{8pt}1 \end{matrix}\right]$} $\in\aut(\frakg_c)$. Then
 $C^tB^t [h] BC =$ {\footnotesize$\left[\begin{matrix}1 & 0 & 0\\ 0 & 0 & 1\\ 0 & 1 & 0  \end{matrix}\right]$}.

Now assume that  $h_{11}h_{22}-h_{12}^2= 0$ and $h_{12}\ne 0$. Then $h_{11}\ne0, h_{22}\ne 0$.
Let  $B =${\footnotesize
$\left[\begin{matrix}
\frac{h_{12}}{h_{11}} & 0 & 0\\
0 & 1 & 0 \\
0 & 0 & 1 \end{matrix}\right]$} $\in\aut(\frakg_c)$. Then
 $B^t [h] B =$ {\footnotesize$\left[\begin{matrix}h_{22} & h_{22} & h_{13}'\\ h_{22} & h_{22} & h_{23}'\\ h_{13}' & h_{23}' & h_{33}' \end{matrix}\right]$}.
Since $\det(B^t [h] B)=-h_{22}(h_{13}'-h_{23}')^2<0$, we have
$h_{22}>0$ and $h_{13}'\ne h_{23}'$.
Let $C =\diag\{\frac{1}{\sqrt{h_{22}}},\frac{1}{\sqrt{h_{22}}}, 1\}\in\aut(\frakg_c)$.
Then
$C^t B^t [h] BC =$ {\footnotesize $\left[\begin{matrix}1 & 1 & \nu\\ 1 & 1 & \lambda\\
\nu & \lambda & \mu \end{matrix}\right]$} where $-(\nu-\lambda)^2<0$.
Let $D =${\footnotesize
$\left[\begin{matrix}
1 & 0 & \frac{\nu^2-2\nu\lambda+\mu}{2(\lambda-\nu)}\\
0 & 1 & \frac{\nu^2-\mu}{2(\lambda-\nu)} \\
0 & 0 & 1 \end{matrix}\right]$} $\in\aut(\frakg_c)$. Then
 $D^tC^t B^t [h] BCD =$ {\footnotesize$\left[\begin{matrix}1 & 1 & 0\\ 1 & 1 & \mu\\
0 & \mu & 0\end{matrix}\right]$} where $\mu\ne0$.
Let $E = \diag\{-1,-1,1\}\in\aut(\frakg_c)$.
Then
 $E^tD^tC^t B^t [h] BCDE =$ {\footnotesize$\left[\begin{matrix}1 & \hspace{8pt}1 & \hspace{8pt}0\\ 1 & \hspace{8pt}1 & -\mu\\
0 & -\mu & \hspace{8pt}0\end{matrix}\right]$}.

Next suppose that $h_{11}h_{22}-h_{12}^2\ne 0$.
Then we may assume that $h_{13}=h_{23}=0$ because
$B =$
{\footnotesize $\left[\begin{matrix}
1 & 0 & \frac{h_{13}h_{22}-h_{12}h_{23}}{h_{12}^2-h_{11}h_{22}}\\
0 & 1 & \frac{h_{11}h_{23}-h_{12}h_{13}}{h_{12}^2-h_{11}h_{22}} \\
0 & 0 & 1 \end{matrix}\right]$} $\in\aut(\frakg_c)$ and
$B^t[h]B=[h']$ where $h'_{13}=h'_{23}=0.$

If $h_{12}=0$, then $h_{11}\ne0$ and $h_{22}\ne0$.
Let $B=\diag\{\frac{1}{\sqrt{|h_{11}|}},\frac{1}{\sqrt{|h_{22}|}},1\}\in\aut(\frakg_c)$.
Then
$B^t[h]B$ is of the form
$$
\left[\begin{matrix} 1 & 0 & \hspace{8pt}0\\ 0 & 1 & \hspace{8pt}0\\ 0 & 0 & -\mu \end{matrix}\right], \quad
\left[\begin{matrix} 1 & \hspace{8pt}0 & 0\\ 0 & -1 & 0\\ 0 &\hspace{8pt} 0 & \mu \end{matrix}\right], \quad
\left[\begin{matrix} -1 & 0 & 0\\ \hspace{8pt}0 & 1 & 0\\ \hspace{8pt}0 & 0 & \mu \end{matrix}\right]
$$
where $\mu>0$.

Now assume $h_{12}\ne0$.
When $h_{11}=h_{22}=0$,
$B=\diag\{\frac{1}{h_{12}},1,1\}\in\aut(\frakg_c)$
and $B^t [h] B =$ {\footnotesize$\left[\begin{matrix}0 & 1 & 0\\ 1 & 0 & 0\\ 0 & 0 & \mu \end{matrix}\right]$} where $\mu>0$.
 When $h_{11}=0$ and $h_{22}\ne0$,
  $B =\diag\{\frac{\sqrt{|h_{22}|}}{h_{12}}, \frac{1}{\sqrt{|h_{22}|}}, 1 \}\in\aut(\frakg_c)$
and
$B^t [h] B =$ {\footnotesize $\left[\begin{matrix}0 & \hspace{8pt}1 & 0\\ 1 & \pm 1 & 0\\ 0 & \hspace{8pt}0 & \mu \end{matrix}\right]$} where $\mu>0$.
When $h_{11}\ne0$,
 $B =\diag\{\frac{1}{\sqrt{|h_{11}|}}, \frac{\sqrt{|h_{11}|}}{h_{12}}, 1\}\in\aut(\frakg_c)$
and $B^t[h]B$ is of the form
$$
\left[\begin{matrix} 1 & 1 & 0\\ 1 & \tau & 0\\ 0 & 0 & \nu \end{matrix}\right], \quad
\left[\begin{matrix} -1 & \hspace{8pt}1 & 0\\ \hspace{8pt}1 & -\tau & 0\\\hspace{8pt} 0 & \hspace{8pt}0 & \nu \end{matrix}\right]
$$
where $\nu(\tau-1)<0$.

Suppose $[h]=${\footnotesize $\left[\begin{matrix} -1 & \hspace{8pt}1 & 0\\ \hspace{8pt}1 & -\tau & 0\\ \hspace{8pt}0 & \hspace{8pt}0 & \nu \end{matrix}\right]$} with $1<\tau$ and $\nu<0$.
Let
$$
y_1=x_1, \,\,  y_2=\frac{1}{\sqrt{\tau-1}}(x_1+x_2), \,\, y_3=\frac{1}{\sqrt{-\nu}}x_3
$$
Then $\calB=\{y_1,y_2,y_3\}$ is a basis of the Lie algebra of $G_c$ such that
$[h]_{\calB} = \diag\{-1,-1,-1\}$, which is a contradiction.

Finally it is easy to see that any eleven such distinct matrices are not equivalent.
\end{proof}

Let $P=\left[\begin{matrix}\hspace{8pt}\frac{1}{2w}&\frac{1+w}{2w}&0\\-\frac{1}{2w}&\frac{w-1}{2w}&0\\ \hspace{8pt}0&0&1\end{matrix}\right]$.
For each left invariant Lorentzian metric $h$ on $G_c$ with $c<1$ whose associated matrix is given in Theorem~\ref{thm: c<1}, we obtain an orthonormal basis $\{y_i\}$ as follows:
\begin{itemize}
  \item [(1)] If $[h]=P^{-1}${\footnotesize$\left[\begin{matrix} 0 & 0 & 1\\ 0 & 1 & 0\\ 1 & 0 & 0 \end{matrix} \right]$}$P$,
\begin{align*}
  &y_1=x_2, \, y_2=\frac{1}{\sqrt{2}}(x_1+x_3), \,
  y_3=\frac{1}{\sqrt{2}}(x_1-x_3);\\
  &[y_1,y_2]=[y_3,y_1]=\frac{w-1}{\sqrt{2}}y_1, \,
  [y_2,y_3]=\frac{1+w}{\sqrt{2}}(y_2+y_3).
\end{align*}

  \item [(2)] If $[h]=P^{-1}${\footnotesize$ \left[\begin{matrix} 1 & 0 & 0\\ 0 & 0 & 1\\ 0 & 1 & 0 \end{matrix} \right]$}$P$,
\begin{align*}
  &y_1=x_1, \, y_2=\frac{1}{\sqrt{2}}(x_2+x_3), \, y_3=\frac{1}{\sqrt{2}}(x_2-x_3);\\
  &[y_1,y_2]=[y_3,y_1]=-\frac{1+w}{\sqrt{2}}y_1, \,
  [y_2,y_3]=\frac{1-w}{\sqrt{2}}(y_2+y_3).
\end{align*}

  \item [(3)] If $[h]=P^{-1}${\footnotesize$\left[\begin{matrix} 1 & 1 & 0\\ 1 & 1 & \mu\\ 0 & \mu & 0 \end{matrix} \right]$}$P$ with $\mu>0$,
\begin{align*}
&y_1=x_1, \, y_2=x_1-x_2-\frac{1}{2\mu}x_3, \, y_3=x_1-x_2+\frac{1}{2\mu}x_3;\\
&[y_1,y_2]=[y_3,y_1]=\frac{1+w}{\mu}y_1, \,
[y_2,y_3]=-\frac{2w}{\mu}y_1+\frac{w-1}{2\mu}(y_2+y_3).
\end{align*}

 \item [(4)] If $[h]=P^{-1}${\footnotesize$\left[\begin{matrix} 1 & 0 & \hspace{8pt}0\\ 0 & 1 & \hspace{8pt}0\\ 0 & 0 & -\mu \end{matrix} \right]$}$P$ with $\mu>0$,
\begin{align*}
&y_1=x_1, \, y_2=x_2, \, y_3=\frac{1}{\sqrt{\mu}}x_3;\\
&[y_1,y_2]=0, \,
[y_2,y_3]=\frac{w-1}{\sqrt{\mu}}y_2, \,
[y_3,y_1]=\frac{1+w}{\sqrt{\mu}}y_1.
\end{align*}

  \item [(5)] If $[h]=P^{-1}${\footnotesize$\left[\begin{matrix} 1 & \hspace{8pt}0 & 0\\ 0 & -1 & 0\\ 0 & \hspace{8pt}0 & \mu \end{matrix} \right]$}$P$ with $\mu>0$,
\begin{align*}
&y_1=x_1, \, y_2=\frac{1}{\sqrt{\mu}}x_3, \, y_3=x_2;\\
&[y_1,y_2]=-\frac{1+w}{\sqrt{\mu}}y_1, \,
[y_2,y_3]=\frac{1-w}{\sqrt{\mu}}y_3, \,
[y_3,y_1]=0.
\end{align*}

  \item [(6)] If $[h]=P^{-1}${\footnotesize$\left[\begin{matrix} -1 & 0 & 0\\ \hspace{8pt}0 & 1 & 0\\ \hspace{8pt}0 & 0 & \mu \end{matrix} \right]$}$P$ with $\mu>0$,
\begin{align*}
&y_1=x_2, \, y_2=\frac{1}{\sqrt{\mu}}x_3, \, y_3=x_1;\\
&[y_1,y_2]= -\frac{1-w}{\sqrt{\mu}}y_1, \,
[y_2,y_3]= \frac{1+w}{\sqrt{\mu}}y_3, \,
[y_3,y_1]= 0.
\end{align*}

 \item [(7)] If $[h]=P^{-1}${\footnotesize$\left[\begin{matrix} 0 & 1 & 0\\ 1 & 0 & 0\\ 0 & 0 & \mu \end{matrix} \right]$}$P$ with $\mu>0$,
\begin{align*}
&y_1=\frac{1}{\sqrt{\mu}}x_3, \,
  y_2=\frac{1}{\sqrt{2}}(x_1+x_2), \,
  y_3=\frac{1}{\sqrt{2}}(x_1-x_2);\\
&[y_1,y_2]=\frac{1}{\sqrt{\mu}}y_2+ \frac{w}{\sqrt{\mu}}y_3, \,
[y_2,y_3]=0,\,
[y_3,y_1]=-\frac{w}{\sqrt{\mu}}y_2 - \frac{1}{\sqrt{\mu}}y_3.
\end{align*}

 \item [(8)] If $[h]=P^{-1}${\footnotesize$\left[\begin{matrix} 0 & 1 & 0\\ 1 & 1 & 0\\ 0 & 0 & \mu  \end{matrix} \right]$}$P$ with $\mu>0$,
\begin{align*}
&y_1=\frac{1}{\sqrt{\mu}}x_3, \,   y_2=x_2,\, y_3=x_1-x_2;\\
&[y_1,y_2]=\frac{1-w}{\sqrt{\mu}}y_2, \,
[y_2,y_3]=0, \,
[y_3,y_1]=-\frac{2w}{\sqrt{\mu}}y_2-\frac{1+w}{\sqrt{\mu}}y_3.
 \end{align*}

 \item [(9)] If $[h]=P^{-1}${\footnotesize$\left[\begin{matrix} 0 & \hspace{8pt}1 & 0\\ 1 & -1 & 0\\ 0 & \hspace{8pt}0 & \mu  \end{matrix} \right]$}$P$ with $\mu>0$,
\begin{align*}
&y_1=\frac{1}{\sqrt{\mu}} x_3, \,  y_2=x_1+x_2, \,  y_3=x_2;\\
&[y_1,y_2]=\frac{1+w}{\sqrt{\mu}}y_2-\frac{2w}{\sqrt{\mu}}y_3, \,
[y_2,y_3]=0, \,
[y_3,y_1]=\frac{w-1}{\sqrt{\mu}}y_3.
\end{align*}

  \item [(10)] If $[h]=P^{-1}${\footnotesize$\left[\begin{matrix} 1 & 1 & 0\\ 1 & \tau & 0\\ 0 & 0 & \nu \end{matrix} \right]$}$P$ with $\nu(\tau-1)<0$,\\
\begin{itemize}
  \item[(10-1)] when $\nu>0$ and $\tau<1$,
\begin{align*}
&y_1=x_1, \,  y_2=\frac{1}{\sqrt{\nu}}x_3, \,  y_3=\frac{1}{\sqrt{1-\tau}}(x_1-x_2);\\
&[y_1,y_2]=-\frac{1+w}{\sqrt{\nu}}y_1, \,
[y_2,y_3]=\frac{2w}{\sqrt{\nu(1-\tau)}}y_1 + \frac{1-w}{\sqrt{\nu}}y_3, \,
[y_3,y_1]=0
 \end{align*}

  \item[(10-2)] when $\nu<0$ and $\tau>1$,
\begin{align*}
&y_1=x_1, \,  y_2=\frac{1}{\sqrt{\tau-1}}(x_1-x_2), \,  y_3=\frac{1}{\sqrt{-\nu}}x_3;\\
&[y_1,y_2]=0, \,
[y_2,y_3]=\frac{-2w}{\sqrt{\nu(1-\tau)}}y_1 + \frac{w-1}{\sqrt{-\nu}}y_2, \,
[y_3,y_1]=\frac{1+w}{\sqrt{-\nu}}y_1
\end{align*}
\end{itemize}

  \item [(11)] If $[h]=P^{-1}${\footnotesize$\left[\begin{matrix} -1 & \hspace{8pt}1 & 0\\ \hspace{8pt}1 & -\eta & 0\\ \hspace{8pt}0 & \hspace{8pt}0 & \mu \end{matrix} \right]$}$P$ with $\mu>0$, $\eta<1$,
\begin{align*}
& y_1=\frac{1}{\sqrt{\mu}}x_3, \,  y_2= \frac{1}{\sqrt{1-\eta}}(x_1+x_2),\,  y_3= x_1;\\
&[y_1,y_2]=\frac{1-w}{\sqrt{\mu}}y_2+\frac{2w}{\sqrt{\mu(1-\eta)}}y_3,\,
[y_2,y_3]= 0, \, [y_3,y_1]=-\frac{1+w}{\sqrt{\mu}}y_3.
\end{align*}
\end{itemize}

With respect to the orthonormal basis $\{y_i\}$, we have the following:

\begin{Thm}\label{thm: c<1-curvature}
The Ricci operator $\Ric$ of the metric $h$ on $G_c$ with $c<1$ is expressed as follows: Let $w=\sqrt{1-c}$.
\begin{itemize}
  \item [(1)] If $[h]=\left[\begin{matrix}\hspace{8pt}\frac{1}{2w}&\frac{1+w}{2w}&0\\-\frac{1}{2w}&\frac{w-1}{2w}&0\\ \hspace{8pt}0&0&1\end{matrix}\right]^{-1}\left[\begin{matrix} 0 & 0 & 1\\ 0 & 1 & 0\\ 1 & 0 & 0 \end{matrix} \right]\left[\begin{matrix}\hspace{8pt}\frac{1}{2w}&\frac{1+w}{2w}&0\\-\frac{1}{2w}&\frac{w-1}{2w}&0\\ \hspace{8pt}0&0&1\end{matrix}\right]$,
\begin{align*}
[\Ric]_{\{y_i\}}
&= \left[\begin{matrix} 0 & 0 & 0 \\ 0 & -w(w-1) & w(w-1) \\
0 & -w(w-1) & w(w-1) \end{matrix} \right]
\end{align*}
$\Ric$ is of O'Neill type $\{11,1\}$ or $\{21\}$  when $w=1$ or $w\ne1$, respectively.

  \item [(2)] If $[h]=\left[\begin{matrix}\hspace{8pt}\frac{1}{2w}&\frac{1+w}{2w}&0\\-\frac{1}{2w}&\frac{w-1}{2w}&0\\ \hspace{8pt}0&0&1\end{matrix}\right]^{-1}\left[\begin{matrix} 1 & 0 & 0\\ 0 & 0 & 1\\ 0 & 1 & 0 \end{matrix} \right]\left[\begin{matrix}\hspace{8pt}\frac{1}{2w}&\frac{1+w}{2w}&0\\-\frac{1}{2w}&\frac{w-1}{2w}&0\\ \hspace{8pt}0&0&1\end{matrix}\right]$,
 \begin{align*}
[\Ric]_{\{y_i\}}
&= \left[\begin{matrix} 0 & 0 & 0 \\ 0 & -w(w+1) & w(w+1) \\
0 & -w(w+1) & w(w+1) \end{matrix} \right] \quad(\text{O'Neill type $\{21\}$})
\end{align*}

  \item [(3)] If $[h]=\left[\begin{matrix}\hspace{8pt}\frac{1}{2w}&\frac{1+w}{2w}&0\\-\frac{1}{2w}&\frac{w-1}{2w}&0\\ \hspace{8pt}0&0&1\end{matrix}\right]^{-1}\left[\begin{matrix} 1 & 1 & 0\\ 1 & 1 & \mu\\ 0 & \mu & 0 \end{matrix} \right]\left[\begin{matrix}\hspace{8pt}\frac{1}{2w}&\frac{1+w}{2w}&0\\-\frac{1}{2w}&\frac{w-1}{2w}&0\\ \hspace{8pt}0&0&1\end{matrix}\right]$ with $\mu>0$,
 \begin{align*}
[\Ric]_{\{y_i\}}
&= \left[\begin{matrix} -\frac{2w^2}{\mu^2} & \frac{w(1+w)}{\mu^2} & -\frac{w(1+w)}{\mu^2} \\
\frac{w(1+w)}{\mu^2} & \frac{w(3w-1)}{2\mu^2} & \frac{w(1+w)}{2\mu^2} \\
\frac{w(1+w)}{\mu^2} & -\frac{w(1+w)}{2\mu^2} & \frac{z(1+5w)}{2\mu^2}\end{matrix} \right]
\end{align*}
$\Ric$ is of O'Neill type $\{11,1\}$ or $\{21\}$  when $w=1$ or $w\ne1$, respectively.

  \item [(4)] If $[h]=\left[\begin{matrix}\hspace{8pt}\frac{1}{2w}&\frac{1+w}{2w}&0\\-\frac{1}{2w}&\frac{w-1}{2w}&0\\ \hspace{8pt}0&0&1\end{matrix}\right]^{-1}\left[\begin{matrix} 1 & 0 & \hspace{8pt}0\\ 0 & 1 & \hspace{8pt}0\\ 0 & 0 & -\mu \end{matrix} \right]\left[\begin{matrix}\hspace{8pt}\frac{1}{2w}&\frac{1+w}{2w}&0\\-\frac{1}{2w}&\frac{w-1}{2w}&0\\ \hspace{8pt}0&0&1\end{matrix}\right]$ with $\mu>0$,
\begin{align*}
[\Ric]_{\{y_i\}}
&=\diag\{\tfrac{2(1+w)}{\mu}, \tfrac{2(1-w)}{\mu}, \tfrac{2(1+w^2)}{\mu} \}
\quad(\text{O'Neill type $\{11,1\}$})
\end{align*}

  \item [(5)] If $[h]=\left[\begin{matrix}\hspace{8pt}\frac{1}{2w}&\frac{1+w}{2w}&0\\-\frac{1}{2w}&\frac{w-1}{2w}&0\\ \hspace{8pt}0&0&1\end{matrix}\right]^{-1}\left[\begin{matrix} 1 & \hspace{8pt}0 & 0\\ 0 & -1 & 0\\ 0 & \hspace{8pt}0 & \mu \end{matrix} \right]\left[\begin{matrix}\hspace{8pt}\frac{1}{2w}&\frac{1+w}{2w}&0\\-\frac{1}{2w}&\frac{w-1}{2w}&0\\ \hspace{8pt}0&0&1\end{matrix}\right]$ with $\mu>0$,
\begin{align*}
[\Ric]_{\{y_i\}}
&=\diag\{-\tfrac{2(1+w)}{\mu}, -\tfrac{2(1+w^2)}{\mu}, \tfrac{2(1-w)}{\mu} \}
\quad(\text{O'Neill type $\{11,1\}$})
\end{align*}

  \item [(6)] If $[h]=\left[\begin{matrix}\hspace{8pt}\frac{1}{2w}&\frac{1+w}{2w}&0\\-\frac{1}{2w}&\frac{w-1}{2w}&0\\ \hspace{8pt}0&0&1\end{matrix}\right]^{-1}\left[\begin{matrix} -1 & 0 & 0\\ \hspace{8pt}0 & 1 & 0\\ \hspace{8pt}0 & 0 & \mu \end{matrix} \right]\left[\begin{matrix}\hspace{8pt}\frac{1}{2w}&\frac{1+w}{2w}&0\\-\frac{1}{2w}&\frac{w-1}{2w}&0\\ \hspace{8pt}0&0&1\end{matrix}\right]$ with $\mu>0$,
\begin{align*}
[\Ric]_{\{y_i\}}
&=\diag\{\tfrac{2(w-1)}{\mu}, -\tfrac{2(1+w^2)}{\mu}, -\tfrac{2(1+w)}{\mu}\}
\quad(\text{O'Neill type $\{11,1\}$})
\end{align*}

 \item [(7)] If $[h]=\left[\begin{matrix}\hspace{8pt}\frac{1}{2w}&\frac{1+w}{2w}&0\\-\frac{1}{2w}&\frac{w-1}{2w}&0\\ \hspace{8pt}0&0&1\end{matrix}\right]^{-1}\left[\begin{matrix} 0 & 1 & 0\\ 1 & 0 & 0\\ 0 & 0 & \mu \end{matrix} \right]\left[\begin{matrix}\hspace{8pt}\frac{1}{2w}&\frac{1+w}{2w}&0\\-\frac{1}{2w}&\frac{w-1}{2w}&0\\ \hspace{8pt}0&0&1\end{matrix}\right]$ with $\mu>0$,
\begin{align*}
[\Ric]_{\{y_i\}}
&=\diag\{-\tfrac{2}{\mu},-\tfrac{2}{\mu},-\tfrac{2}{\mu}\}
\quad(\text{O'Neill type $\{11,1\}$})\\
\end{align*}

\item [(8)] If $[h]=\left[\begin{matrix}\hspace{8pt}\frac{1}{2w}&\frac{1+w}{2w}&0\\-\frac{1}{2w}&\frac{w-1}{2w}&0\\ \hspace{8pt}0&0&1\end{matrix}\right]^{-1}\left[\begin{matrix} 0 & 1 & 0\\ 1 & 1 & 0\\ 0 & 0 & \mu  \end{matrix} \right]\left[\begin{matrix}\hspace{8pt}\frac{1}{2w}&\frac{1+w}{2w}&0\\-\frac{1}{2w}&\frac{w-1}{2w}&0\\ \hspace{8pt}0&0&1\end{matrix}\right]$ with $\mu>0$,
\begin{align*}
[\Ric]_{\{y_i\}}
&= \left[\begin{matrix} -\frac{2}{\mu} & 0 & 0 \\
0 & -\frac{2(w^2-w+1)}{\mu} & \frac{2w(w-1)}{\mu} \\
0 & -\frac{2w(w-1)}{\mu} & \frac{2(w^2-w-1)}{\mu} \end{matrix} \right]
\end{align*}
$\Ric$ is of O'Neill type $\{11,1\}$ or $\{21\}$  when $w=1$ or $w\ne1$, respectively.

\item [(9)] If $[h]=\left[\begin{matrix}\hspace{8pt}\frac{1}{2w}&\frac{1+w}{2w}&0\\-\frac{1}{2w}&\frac{w-1}{2w}&0\\ \hspace{8pt}0&0&1\end{matrix}\right]^{-1}\left[\begin{matrix} 0 & \hspace{8pt}1 & 0\\ 1 & -1 & 0\\ 0 & \hspace{8pt}0 & \mu  \end{matrix} \right]\left[\begin{matrix}\hspace{8pt}\frac{1}{2w}&\frac{1+w}{2w}&0\\-\frac{1}{2w}&\frac{w-1}{2w}&0\\ \hspace{8pt}0&0&1\end{matrix}\right]$ with $\mu>0$,
\begin{align*}
[\Ric]_{\{y_i\}}
&= \left[\begin{matrix} -\frac{2}{\mu} & 0 & 0 \\
\hspace{8pt}0 & \frac{2(w^2-w-1)}{\mu} & \frac{2w(w-1)}{\mu} \\
\hspace{8pt}0 & -\frac{2w(w-1)}{\mu} & -\frac{2(w^2-w+1)}{\mu} \end{matrix} \right]
\end{align*}
$\Ric$ is of O'Neill type $\{11,1\}$ or $\{21\}$  when $w=1$ or $w\ne1$, respectively.

 \item [(10-1)] If $[h]=\left[\begin{matrix}\hspace{8pt}\frac{1}{2w}&\frac{1+w}{2w}&0\\-\frac{1}{2w}&\frac{w-1}{2w}&0\\ \hspace{8pt}0&0&1\end{matrix}\right]^{-1}\left[\begin{matrix} 1 & 1 & 0\\ 1 & \tau & 0\\ 0 & 0 & \nu \end{matrix} \right]\left[\begin{matrix}\hspace{8pt}\frac{1}{2w}&\frac{1+w}{2w}&0\\-\frac{1}{2w}&\frac{w-1}{2w}&0\\ \hspace{8pt}0&0&1\end{matrix}\right]$ with $\nu>0$ and $\tau<1$,\\
\begin{align*}
[\Ric]_{\{y_i\}}
&= \left[\begin{matrix} -\frac{2(w^2+w+1-(1+w)\tau)}{\nu(1-\tau)} & 0 & -\frac{2w(1+w)}{\nu\sqrt{1-\tau}} \\
0 & \frac{2(w^2\tau+\tau-1)}{\nu(1-\tau)} & 0 \\
\frac{2w(1+w)}{\nu\sqrt{1-\tau}} & 0 & \frac{2(w^2+w-1+(1-w)\tau)}{\nu\sqrt{1-\tau}} \end{matrix} \right]
\end{align*}
$\Ric$ is of O'Neill type $\{11,1\}$, $\{1z\bar{z}\}$, $\{21\}$   when $\tau(\tau-1+w^2) >0, <0, =0$, respectively.

 \item [(10-2)] If $[h]=\left[\begin{matrix}\hspace{8pt}\frac{1}{2w}&\frac{1+w}{2w}&0\\-\frac{1}{2w}&\frac{w-1}{2w}&0\\ \hspace{8pt}0&0&1\end{matrix}\right]^{-1}\left[\begin{matrix} 1 & 1 & 0\\ 1 & \tau & 0\\ 0 & 0 & \nu \end{matrix} \right]\left[\begin{matrix}\hspace{8pt}\frac{1}{2w}&\frac{1+w}{2w}&0\\-\frac{1}{2w}&\frac{w-1}{2w}&0\\ \hspace{8pt}0&0&1\end{matrix}\right]$ with $\nu<0$ and $\tau>1$,\\
\begin{align*}
[\Ric]_{\{y_i\}}
&= \left[\begin{matrix} \frac{2(w^2+w+1-(1+w)\tau)}{\nu(\tau-1)} & -\frac{2w(1+w)}{\nu\sqrt{\tau-1}} & 0 \\
-\frac{2w(1+w)}{\nu\sqrt{\tau-1}} & -\frac{2(w^2+w-1+(1-w)\tau)}{\nu(\tau-1)} & 0 \\
0 & 0 & -\frac{2(w^2\tau+\tau-1)}{\nu(\tau-1)}\end{matrix} \right]\\
&(\text{O'Neill type $\{11,1\}$})
\end{align*}

  \item [(11)] If $[h]=\left[\begin{matrix}\hspace{8pt}\frac{1}{2w}&\frac{1+w}{2w}&0\\-\frac{1}{2w}&\frac{w-1}{2w}&0\\ \hspace{8pt}0&0&1\end{matrix}\right]^{-1}\left[\begin{matrix} -1 & \hspace{8pt}1 & 0\\ \hspace{8pt}1 & -\eta & 0\\ \hspace{8pt}0 & \hspace{8pt}0 & \mu \end{matrix} \right]\left[\begin{matrix}\hspace{8pt}\frac{1}{2w}&\frac{1+w}{2w}&0\\-\frac{1}{2w}&\frac{w-1}{2w}&0\\ \hspace{8pt}0&0&1\end{matrix}\right]$ with $\mu>0$, $\eta<1$,
\begin{align*}
[\Ric]_{\{y_i\}}
&= \left[\begin{matrix} \frac{2(w^2\eta+\eta-1)}{\mu(1-\eta)} & 0 & 0 \\
0 & \frac{2(w^2+w-1+\eta-w\eta)}{\mu(1-\eta)} & \frac{2w(1+w)}{\mu\sqrt{1-\eta}} \\
0 & -\frac{2w(1+w)}{\mu\sqrt{1-\eta}} & -\frac{2(w^2+w+1 - (1+w)\eta)}{\mu(1-\eta)} \end{matrix} \right]
\end{align*}
$\Ric$ is of O'Neill type $\{11,1\}$, $\{1z\bar{z}\}$, $\{21\}$   when $\eta(\eta-1+w^2) >0, <0, =0$, respectively.
\end{itemize}
\end{Thm}

\begin{proof}
Applying Lemma~\ref{lemma:OT1} and ~\ref{lemma:OT2},
except for (3), the O'Neill types in all cases can be easily obtained.
(3) Let $[\Ric]_{\{y_i\}}
= \left[\begin{matrix} -\frac{2w^2}{\mu^2} & \frac{w(1+w)}{\mu^2} & -\frac{w(1+w)}{\mu^2} \\
\frac{w(1+w)}{\mu^2} & \frac{w(3w-1)}{2\mu^2} & \frac{w(1+w)}{2\mu^2} \\
\frac{w(1+w)}{\mu^2} & -\frac{w(1+w)}{2\mu^2} & \frac{z(1+5w)}{2\mu^2}\end{matrix} \right]$.
Then $[\Ric]_{\{y_i\}}$  has two eigenvalues $-\frac{2w^2}{\mu^2},\frac{2w^2}{\mu^2}$ ($\frac{2w^2}{\mu^2}$ has multiplicity two).
First assume $w=1$.
If we choose an orthonormal basis $\frakB'$ given by
$$
[\id]_{\frakB',\frakB} =
\left[\begin{matrix} -1& \hspace{8pt}1& \hspace{8pt}1 \\ \hspace{8pt}\frac{1}{2} & \hspace{8pt}1 & \hspace{8pt}\frac{1}{2} \\ \hspace{8pt}\frac{1}{2} & -1 & -\frac{3}{2} \end{matrix}\right]\in O(2,1),
$$
then
$$
[\Ric]_{\frakB'} = \left[\begin{matrix}-\frac{2}{\mu^2}&0&0\\\hspace{8pt}0&\frac{2}{\mu^2}&0\\\hspace{8pt}0&0&\frac{2}{\mu^2}\end{matrix}\right]
$$
Hence $\Ric$ is of O'Neill type $\{11,1\}$.
Now suppose $w\ne 1$. Then the dimension of the eigenspace of the eigenvalue  $\frac{2w^2}{\mu^2}$ is $1$.
Hence $\Ric$ is of O'Neill type $\{21\}$.
In this case, if we choose an orthonormal basis $\frakB'$ given by
$$
[\id]_{\frakB',\frakB} =\left\{\begin{array}{r}
\left[\begin{matrix} 1& \frac{\mu\sqrt{1+w}}{2w\sqrt{w-1}}& -\frac{\mu\sqrt{1+w}}{2w\sqrt{w-1}} \\
-\frac{1+w}{4w} & \frac{4w^2(w^2-1)+(15w^2-2w-1)\mu^2}{16w^2\sqrt{w^2-1}\mu} & \frac{4w^2(w^2-1)-(15w^2-2w-1)\mu^2}{16w^2\sqrt{w^2-1}\mu} \\
-\frac{w+1}{4w} & \frac{4w^2(w^2-1)-(17w^2+2w+1)\mu^2}{16w^2\sqrt{w^2-1}\mu} & \frac{4w^2(w^2-1)+(17w^2+2w+1)\mu^2}{16w^2\sqrt{w^2-1}\mu} \end{matrix}\right],\\ \text{if $w>1$};\\
\left[\begin{matrix} 1& \frac{\mu\sqrt{1+w}}{2w\sqrt{1-w}}& \frac{\mu\sqrt{1+w}}{2w\sqrt{1-w}} \\
-\frac{1+w}{4w} & \frac{4w^2(1-w^2)+(15w^2-2w-1)\mu^2}{16w^2\sqrt{1-w^2}\mu} & \frac{4w^2(w^2-1)+(15w^2-2w-1)\mu^2}{16w^2\sqrt{1-w^2}\mu} \\
-\frac{1+w}{4w} & \frac{4w^2(w^2-1)+(17w^2+2w+1)\mu^2}{16w^2\sqrt{1-w^2}\mu} & \frac{4w^2(w^2-1)-(17w^2+2w+1)\mu^2}{16w^2\sqrt{1-w^2}\mu} \end{matrix}\right],\\
\text{ if $0<w<1$},
\end{array} \right.
$$
then
$$
[\Ric]_{\frakB'} = \left\{\begin{array}{ll} \left[\begin{matrix}-\frac{2w^2}{\mu^2}&0&0\\\hspace{8pt}0&\frac{2w^2}{\mu^2}-1&1\\\hspace{8pt}0&-1&\frac{2w^2}{\mu^2}+1\end{matrix}\right],
&w>1;\\
\left[\begin{matrix}-\frac{2w^2}{\mu^2}&0&0\\\hspace{8pt}0&\frac{2w^2}{\mu^2}+1&1\\\hspace{8pt}0&-1&\frac{2w^2}{\mu^2}-1\end{matrix}\right], &0<w<1
\end{array} \right.
$$
The proof is finished.
\end{proof}

\begin{Rmk}
Three-dimensional Lie groups, having a flat Lorentzian metric, have been classified in \cite{C1} and \cite{Nomizu}.
We classify explicitly three-dimensional Lorentzian non-unimodular Lie groups  whose metrics have constant section curvatures.

Let $(G,h)$ be a connected, simply connected three-dimensional Lorentzian non-unimodular Lie
group $G$.
Then we have:
\begin{enumerate}
  \item $(G,h)$ is flat if and only if $h$ is equivalent up to automorphism to a metric whose associated matrix is of the form
\begin{align*}
 (a)&\left[\begin{matrix} 1 & 0 & 0 \\ 0 & 0 & 1 \\ 0 & 1 & 0 \end{matrix} \right] \text{on $G_I$};\\
 (b)&\left[\begin{matrix}-2&-1&0\\\hspace{8pt}1&\hspace{8pt}1&0\\ \hspace{8pt}0&\hspace{8pt}0&1\end{matrix}\right]^{-1}
  \left[\begin{matrix} 0 & 0 & 1\\ 0 & \mu & 0\\ 1 & 0 & 0 \end{matrix} \right]
  \left[\begin{matrix}-2&-1&0\\\hspace{8pt}1&\hspace{8pt}1&0\\ \hspace{8pt}0&\hspace{8pt}0&1\end{matrix}\right]
  \text{with $\mu>0$ on $G_1$};\\
 (c)&\left[\begin{matrix}\hspace{8pt}\frac{1}{2}&1&0\\-\frac{1}{2}&0&0\\\hspace{8pt}0&0&1\end{matrix}\right]^{-1}
\left[\begin{matrix} 0 & 0 & 1\\ 0 & 1 & 0\\ 1 & 0 & 0 \end{matrix}\right]
\left[\begin{matrix}\hspace{8pt}\frac{1}{2}&1&0\\-\frac{1}{2}&0&0\\\hspace{8pt}0&0&1\end{matrix}\right]
\text{on $G_0$.}
\end{align*}

  \item $(G,h)$ has positive constant sectional curvature if and only if
  $h$ is equivalent up to automorphism to a metric whose associated matrix is of the form
\begin{align*}
 &\left[\begin{matrix} 1 & 0 & \hspace{8pt}0 \\ 0 & 1 & \hspace{8pt}0 \\ 0 & 0 & -\mu \end{matrix} \right] \text{with $\mu>0$ on $G_I$.}
\end{align*}

  \item $(G,h)$ has negative constant sectional curvature if and only if
  $h$ is equivalent up to automorphism to a metric whose associated matrix is of the form
\begin{align*}
(a)&\left[\begin{matrix} 1 & \hspace{8pt}0 & 0 \\ 0 & -1 & 0 \\ 0 & \hspace{8pt}0 & \mu \end{matrix} \right] \text{with $\mu>0$ on $G_I$};\\
(b)&\left[\begin{matrix}\hspace{8pt}\frac{1}{2w}&\frac{1+w}{2w}&0\\-\frac{1}{2w}&\frac{w-1}{2w}&0\\ \hspace{8pt}0&0&1\end{matrix}\right]^{-1}\left[\begin{matrix} 0 & 1 & 0\\ 1 & 0 & 0\\ 0 & 0 & \mu \end{matrix} \right]\left[\begin{matrix}\hspace{8pt}\frac{1}{2w}&\frac{1+w}{2w}&0\\-\frac{1}{2w}&\frac{w-1}{2w}&0\\ \hspace{8pt}0&0&1\end{matrix}\right] \\
&\text{ with $\mu>0$ on $G_c, c<1$ where $w=\sqrt{1-c}$.}
\end{align*}

\end{enumerate}

\end{Rmk}


\begin{thebibliography}{99}

\bibitem{BC}
M.~Boucetta and A.~Chakkar,
{\it The moduli spaces of Lorentzian left-invariant metrics on three-dimensional unimodular simply connected Lie groups}, to appear in J. Korean Math. Soc.


\bibitem{C1}
G.~Calvauso,
{\it Homogeneous structures on three-dimensional Lorentzian manifolds}, J. Geom. Phys., {\bf 57} (2007), 1279–-1291.

\bibitem{HL2009_MN}
K.~Y.~Ha and J.~B.~Lee,
{\it Left invariant metrics and curvatures on connected, simply connected three-dimensional Lie groups},
Math. Nachr., {\bf 282} (2009), 868--898.

\bibitem{Milnor}
J.~Milnor,
{\it Curvatures of left-invariant metrics on Lie groups},
Adv. Math., {\bf 21} (1976), 293--329.

\bibitem{Nomizu}
K.~Nomizu,
{\it Left-invariant Lorentz metrics on Lie groups},
Osaka J. Math., {\bf 16} (1979), 143-–150.


\bibitem{ONeill}
B.~O'Neill, {\it Semi-Riemannian Geometry with applications to relativity},
{Academic Press}, 1983.

\end{thebibliography}
\end{document}